\documentclass[11pt]{article}

\usepackage{times}
\usepackage{amsmath,amsfonts,amstext,amssymb,amsbsy,amsopn,amsthm,eucal}
\usepackage{txfonts}
\usepackage{dsfont}
\usepackage{graphicx}   
\usepackage{hyperref}

\hypersetup{
  pdftitle={Quantitative regularity for p-harmonic maps},
  pdfauthor={Aaron Naber and Jona Veronelli and Daniele Valtorta},
  pdfsubject={Quantitative regularity for p-harmonic maps},
  pdfkeywords={},
  pdfpagelayout=SinglePage,
  pdfpagemode=UseOutlines,
  colorlinks,
  bookmarksopen,
  linkcolor=[rgb]{0,0,0.7},
  urlcolor=[rgb]{0,0,0.4},
  citecolor=[rgb]{0.4,0.1,0}
}

\usepackage{tocbibind}

\newcommand{\Vol}{\text{Vol}}
\newcommand{\diam}{\text{diam}}

\newcommand{\N}{\mathbb{N}}

\newcommand{\R}{\mathbb{R}}

\newcommand{\cB}{\mathcal{B}}
\newcommand{\cC}{\mathcal{C}}

\newcommand{\cF}{\mathcal{F}}

\newcommand{\cH}{\mathcal{H}}

\newcommand{\cM}{\mathcal{M}}

\newcommand{\cS}{\mathcal{S}}

\newcommand{\cW}{\mathcal{W}}

\newcommand{\floor}[1]{\lfloor #1 \rfloor}
\newcommand{\ceil}[1]{\lceil #1 \rceil}

\newcommand{\norm}[1]{\left\|#1\right\|}

\newcommand{\ps}[2]{\left\langle#1\middle\vert#2\right\rangle}

\newcommand{\ton}[1]{\left(#1\right)}
\newcommand{\qua}[1]{\left[#1\right]}
\newcommand{\cur}[1]{\left\{#1\right\}}
\newcommand{\abs}[1]{\left|#1\right|}
\newcommand{\wto}{\rightharpoonup}

\newcommand{\dive}[1]{\operatorname{div}\ton{#1}}

\begin{document}

\newtheorem{theorem}{Theorem}[section]
\newtheorem{ctheorem}[theorem]{Conjectural Theorem}

\newtheorem{proposition}[theorem]{Proposition}

\newtheorem{lemma}[theorem]{Lemma}
\newtheorem{clemma}[theorem]{Conjectural Lemma}

\newtheorem{corollary}[theorem]{Corollary}

\theoremstyle{definition}
\newtheorem{definition}[theorem]{Definition}

\theoremstyle{remark}
\newtheorem{remark}[theorem]{Remark}

\theoremstyle{remark}
\newtheorem{example}[theorem]{Example}

\theoremstyle{remark}
\newtheorem{note}[theorem]{Note}

\theoremstyle{definition}
\newtheorem{notation}[theorem]{Notation}

\theoremstyle{remark}
\newtheorem{question}[theorem]{Question}

\theoremstyle{remark}
\newtheorem{conjecture}[theorem]{Conjecture}

\title{Quantitative regularity for p-harmonic maps}

\author{Aaron Naber\footnote{Northwestern University (USA) anaber@math.northwestern.edu}, Daniele Valtorta\footnote{EPFL (CH) daniele.valtorta@epfl.ch} and Giona Veronelli \footnote{Universit\'e Paris 13, Sorbonne Paris Cit\'e, LAGA, CNRS (UMR 7539) (EU) veronelli@math.univ-paris13.fr}}

\date{\today}
\maketitle
\begin{abstract}
In this article, we study the regularity of minimizing and stationary $p$-harmonic maps between Riemannian manifolds. The aim is obtaining Minkowski-type volume estimates on the singular set $\cS(u)=\cur{x \ \ s.t. \ \ u \text{ is not continuous at } x}$, as opposed to the weaker and non quantitative Hausdorff dimension bounds currently available in literature for generic $p$. 

The main technique used in this paper is the quantitative stratification, which is based on the study of the approximate symmetries of the tangent maps of $u$. In this article, we generalize the study carried out in \cite{ChNa2} for minimizing $2$-harmonic maps to generic $p\in (1,\infty)$. Moreover, we analyze also the stationary case where the lack of compactness makes the study more complicated.

In order to understand the degeneracy intrinsic in the behaviour of stationary maps, we study the defect measure naturally associated to a sequence of such maps and generalize the results obtained in \cite{lin_stat}.

By using refined covering arguments, we also improve the estimates in the case of isolated singularities and obtain a definite bound on the number of singular points. This result seems to be new even for minimizing $2$-harmonic maps.
\end{abstract}

\tableofcontents

\clearpage

\section{Introduction}
In this article, we study the regularity of minimizing and stationary $p$-harmonic maps between Riemannian manifolds, for $p\in (1,\infty)$. That is, given two compact Riemannian manifolds $M$ and $N$, where $N$ has empty boundary, we consider the critical points of the functional
\begin{gather*}
 E_p(u)= \int_M \abs{\nabla u}^p\, ,
\end{gather*}
and focus on the local regularity of $u$. The singular set of such a function is defined as 
\begin{gather*}
 \cS(u)=\cur{x\in M \ \ s.t. \ \ u \text{ is not continuous at } x}\, .
\end{gather*}
Similar to the $2$-harmonic case \cite{lin_stat} we will introduce the notion of a defect measure for limits of such mappings.  We will use this in conjunction with the quantitative stratification technique to prove effective Minkowski-type estimates not only on $\cS(u)$, but also on the regularity scale of $u$ (see Definition \ref{deph_regscale}), which roughly speaking controls the regularity of $u$ in a neighborhood of every point. As a corollary we obtain sharp integrability conditions for $\nabla u$.  See Theorems \ref{th_main} and \ref{th_main_m} for complete statements.

\subsection{Definitions and Notation} For the reader's convenience, we recall the standard definitions of $p$-harmonic maps. Let $M$ and $N$ be two smooth compact Riemannian manifolds, $N$ without boundary, and $M$ of dimension $m$. We will always assume that $N$ is isometrically embedded in some Euclidean space $\R^n$ (note that $n$ is not the dimension of $N$), and we will denote by $W^{1,p}(M,N)$ the Sobolev space of maps $u\in W^{1,p}(M,\R^n)$ such that $u(x)\in N$ a.e. in $M$. A map $u\in W^{1,p}(M,N)$ is said to be a weakly $p$-harmonic map if it (weakly) satisfies the equation
\begin{gather*}
 \Delta_p(u)=\dive{\abs{\nabla u}^{p-2} \nabla u} = -\abs{\nabla u}^{p-2} II(u)(\nabla u,\nabla u)\, ,
\end{gather*}
where $II$ is the second fundamental form of $N$. Equivalently, such a map has the property that for every smooth vector field $\xi:M\to \R^n$ with compact support
\begin{gather*}
 \left.\frac{d}{dt}\right\vert_{t=0} \int_M \abs{\nabla \ton{\Pi(u+t\xi)}}^p =0\, ,
\end{gather*}
where $\Pi$ is the nearest point projection on $N$ defined on a tubular neighborhood of the manifold inside the ambient Euclidean space. If in addition, $u$ is a critical point with respect to variations in its domain of definition, then it is called a \textit{stationary} $p$-harmonic map. In particular, a stationary map is a weakly $p$-harmonic map satisfying
 \begin{gather*}
  \left.\frac{d}{dt}\right\vert_{t=0} E_p(u(\exp_{x}(t\chi(x))))=\left.\frac{d}{dt}\right\vert_{t=0} \int_{M} \abs{\nabla u(\exp_{x}(t\chi(x)))}^p dV =0\, 
 \end{gather*}
for all smooth compactly supported $\chi:M\to TM$. Here by $\exp_{x}(\cdot)$ we mean the exponential map centered at $x$ which sends $T_x(M)$ into $M$. If $M\subset \R^m$, then evidently $\exp_{x}(t\chi(x))=x+t\chi(x)$. Note that a weakly $p$-harmonic map in $C^1(M,N)$ is necessarily stationary.
 
Finally, we define $u$ to be a minimizing $p$-harmonic map, or more simply a $p$-minimizer, if $u$ minimizes the $p$-energy in the class of $W^{1,p}$ maps with the same trace on $\partial M$.

An important tool in the study of such maps is the normalized energy, defined as
\begin{gather*}
 \theta_u(x,r)= r^{p-m} \int_{B_r(x)} \abs{\nabla u}^p\, .
\end{gather*}
This quantity turns out to be monotone (or almost monotone) for stationary maps.

Throughout the paper, we will use the standard notation $\floor p$ to denote the integral part of a real number, i.e., the biggest integer $\leq p$.

\subsection{Background}
The regularity of $p$-harmonic maps has been extensively studied in literature, in particular when $p=2$. One should also be careful in separating the minimizing and the stationary case. Note that by Sobolev embedding $u$ is continuous if $p>m$, making $p\leq m$ the only interesting case.

In \cite{SU} it was proved that the singular set $\cS(u)$ for $2$-minimizers has Hausdorff dimension at most $m-3$, and outside the singular set the map $u$ is actually smooth. Their proof is based on a dimension reduction argument and on an important $\epsilon$-regularity theorem according to which if $\theta(x,2r)<\epsilon(m,N)$, then $u$ is smooth on $B_{r}(x)$. Additionally, under the additional assumption that there exist no continuous $2$-minimizers $u:S^i\to N$ for $i=2,\cdots,k$, they can improve the Hausdorff dimension estimates to $m-k-2$.  

For generic $p\neq 2$, the situation is similar, although in this case the lack of uniform ellipticity makes $C^{1,\alpha}$ estimates the best regularity one can hope for, as opposed to smooth estimates. Indeed, in \cite{pHL} the authors extend the $\epsilon$-regularity theorem to this case and prove that $\cS(u)$ is a set of Hausdorff dimension $\leq m-\floor p -1$ outside of which $u$ is $C^{1,\alpha}$.

More recently, in \cite{ChNa2} the Hausdorff dimension estimates of \cite{SU} were improved to Minkowski dimension estimates in the $p=2$ case.  Indeed, the estimates of \cite{ChNa2} allow for the first $L^q$ estimates on the gradient and Hessian of solutions to be proved, and more importantly the first $L^q$ estimates on the regularity scale of solutions.  In particular, given a $2$-minimizer $u:B_2(0)\to N$ with $\int_{B_2(0)} \abs{\nabla u}^2 \leq \Lambda$, \cite{ChNa2} shows that for every $\epsilon>0$
\begin{gather}\label{eq_1}
 \Vol(B_r(\cS(u) \cap B_{1}(0)))\leq C(m,N,\Lambda,\epsilon)r^{3-\epsilon}\, .
\end{gather}
The key new ingredient for the proof in \cite{ChNa2} was the introduction of the quantitative stratification.

The goal of this paper is to introduce the quantitative stratification techniques to the generic $p$ context, and to use these results to prove similar effective estimates for $p$-harmonic maps between Riemannian manifolds.  To do this it will be necessary for us to develop the notion of a defect measure, which will allow us to study limits of $p$-harmonic maps.

Indeed, note that many arguments in the proofs of these results rely on some compactness properties enjoyed by the family of $p$-minimizers.  That is, if a sequence $u_i$ of $p$-minimizers converges weakly in the $W^{1,p}$ sense to some $u$, then the convergence is actually strong and $u$ is a $p$-minimizer (see \cite{luck} or \cite[section 2.9]{simon}).

Stationary maps do not enjoy this compactness property, and thus are in general worse behaved than minimizing ones.  Regardless Bethuel proved in \cite{beth} an $\epsilon$-regularity theorem for stationary $2$-harmonic maps.  This makes it possible to estimate that $\cH^{m-2}(\cS(u))=0$. A sharp estimate in this case is still an interesting open problem.

The technique used by Bethuel is difficult to generalize for arbitrary $p$, and in fact a full-blown $\epsilon$-regularity theorem is not available in literature. To the best of our knowledge, the most general result is the one in \cite{torowang}, which assumes that the target space $N$ is a homogeneous space with a left invariant metric. In this case, the authors are able to generalize the $\epsilon$-regularity theorem and obtain as a corollary that $\cH^{m-p}(\cS(u))=0$, where, as in the minimizing case, $u$ is $C^{1,\alpha}$ outside of its singular set. 

Just as Bethuel's result, this result is based on the duality between BMO and Hardy spaces, and on a special choice of gauge which allow to exploit this duality to conclude a polynomial decay for $\theta(x,r)$ when $\theta(x,1)\leq \epsilon$. However, when $p\neq 2$, finding this gauge presents nontrivial technical difficulties, which are easily overcome if the target space has some special structure.

Note that similar results are available when $N$ is a round sphere, see for example \cite{p-fuchs}, \cite{take_eps}, \cite{strze_eps}, \cite{mouya}, \cite{strze_eps_revenge}, \cite{Rtoc}.

Regarding the lack of compactness for stationary maps, an interesting study has been carried out in \cite{lin_stat} when $p=2$. Given a $W^{1,2}$ weakly convergent sequence of stationary maps $u_i\wto u$, one can define
\begin{gather*}
 \abs{\nabla u_i}^2dx \wto \abs{\nabla u}^2dx +\nu\, ,
\end{gather*}
where the convergence is in the weak sense of measures. The nonnegative measure $\nu$ is the \textit{defect measure}, and it is clear that $u_i$ converges strongly in $W^{1,2}$ to $u$ if and only if $\nu$ is null. In \cite{lin_stat}, the author studies the measure-theoretical properties of the measure $\nu$, focusing in particular on its relation with the $n-2$ Hausdorff measure and its rectifiability, and via dimension reduction arguments he is able to prove that if such a measure exists, then there exists also a smooth nonconstant stationary $2$-harmonic map $h:S^2\to N$. Thus in case such a map did not exist, stationary maps would enjoy the same compactness properties of minimizers, and thus also the same regularity properties. This fact is used in \cite[corollary 1.26]{chhana} to prove an estimate similar to \eqref{eq_1} for $2$-stationary maps.

In this paper we will similarly introduce the defect measure for limits of stationary $p$-harmonic maps, and we will see it enjoys all the same properties enjoyed by the defect measure for $2$-harmonic maps.  We will use it as in \cite[corollary 1.26]{chhana} to give regularity estimates for some stationary harmonic maps.


\subsection{Main Results}
In this article, we generalize the quantitative stratification technique introduced in \cite{ChNa2} to generic $p\in (1,\infty)$, and use it to obtain regularity estimates for both minimizers and stationary maps.  To do this we introduce and study the {\it defect measure} associated to a sequence of stationary $p$-harmonic maps.

For the sake of convenience, we will assume that the base manifold $M$ is a smooth Riemannian manifold with
\begin{gather}\label{eq_Mbounds}
 \abs{\operatorname{sec}(M)}\leq 1 \, , \quad \quad \operatorname{inj}(M)\geq 2\, .
\end{gather}

Before stating the results, we define two conditions on the target manifold $N$ under which we will be able to obtain improved regularity results.

\begin{definition}
 We say that a compact manifold $N$ satisfies condition \eqref{A} if
\begin{gather}\label{A}\tag{A}
 \not \exists \text{ nonconstant continuous $p$-minimizing maps  } u:S^i\to N \ \ \ i=\floor p, \cdots, a\, .
\end{gather}
 We say that a compact manifold $N$ satisfies condition \eqref{B} if
\begin{gather}\label{B}\tag{B}
 \not \exists \text{ nonconstant continuous $p$-stationary maps  } u:S^i\to N \ \ \ i=\floor p, \cdots, b\, .
\end{gather} 


\end{definition}

%

\subsubsection{Results for Minimizers}

In the minimizing case, by combining the quantitative stratification with the $\epsilon$-regularity theorem in \cite{pHL} we obtain the following Minkowski-type estimates 
\begin{theorem}\label{th_main}
 Let $u$ be a $p$-minimizing map $u:B_2(0)\subset M\to N$, where $N$ is compact (without boundary) and
 \begin{gather*}
  \int_{B_2(0)}\abs{\nabla u}^p dV \leq \Lambda\, .
 \end{gather*}
If $m\geq \floor{p}+1$, then for every $\eta>0$, there exists a constant $C(m,N,\Lambda,p,\eta)$ such that  for every $r\geq 0$
\begin{gather*}
 \Vol\ton{B_r(\cS(u))\cap B_{1}(0)}\leq C r^{\floor{p}+1-\eta}\, .
\end{gather*}
Under the additional assumption \eqref{A}, we can improve the result to
\begin{gather*}
 \Vol\ton{B_r(\cS(u))\cap B_{1}(0)}\leq C r^{a+2-\eta}\, .
\end{gather*}
\end{theorem}
As a corollary of the proof, we will obtain the following integrability properties.
\begin{corollary}
 Under the hypotheses of the previous theorem, for all $\epsilon>0$, $\nabla u\in L^{\floor{p}+1-\epsilon}(B_1(0))$ with
 \begin{gather*}
  \int_{B_1(0)} \abs{\nabla u}^{\floor{p}+1-\epsilon}\leq C(m,\Lambda,N,p,\epsilon)\, .
 \end{gather*}
Moreover, under the additional assumption \eqref{A}, $\nabla u\in L^{a+2-\epsilon}(B_1(0))$ with
\begin{gather*}
 \int_{B_1(0)} \abs{\nabla u}^{a+2-\epsilon}\leq C(m,\Lambda,N,p,\epsilon)\, .
\end{gather*}
\end{corollary}

In the borderline case $m=\floor p +1$, it is known that the singularities are isolated (see for example \cite{SU,pHL}). Using a refined covering argument, we are able to improve the previous estimate to an effective finiteness of the number of singularities for the map $u$. This result appears to be new even if $p=2$.
\begin{theorem}\label{th_main_m}
 Let $u$ be a $p$-minimizing map $u:B_2(0)\subset M\to N$, where $N$ is compact (without boundary) and
 \begin{gather*}
  \int_{B_2(0)}\abs{\nabla u}^p dV \leq \Lambda\, .
 \end{gather*}
Suppose that $m= \floor{p}+1$ or that, under the additional assumption \eqref{A}, $m=a+2$. Then there exists a constant $C(p,N,\Lambda)$ such that
 \begin{gather*}
  \# \cS(u)\cap B_{1}(0)\leq C(p,\Lambda,N)\, .
 \end{gather*}
\end{theorem}

\begin{remark}
 As it is evident, the lower bound on the injectivity radius and the sectional curvature of the manifold $M$ in \eqref{eq_Mbounds} are arbitrary. Indeed, by scaling and covering it is immediate to see that all the results in this section hold for a generic smooth manifold, up to letting $C$ depend also on the lower bounds on curvature and injectivity radius.
\end{remark}

\begin{remark}
 As mentioned before, the case $m<p$ is not interesting since Sobolev embedding implies immediately Holder continuity, and by standard arguments one gets effective $C^{1,\alpha}$ regularity from it. The borderline case $m=p$ is also not very difficult to deal with (\cite[section 3.6]{simon}). For the reader's convenience, we will briefly sketch a quick self-contained argument to prove these statements in Theorem \ref{th_m<p}.
\end{remark}

\subsubsection{Results for Stationary maps}

As for the stationary case, we will start by generalizing the study of the defect measure in \cite{lin_stat} to a generic $p\in (1,\infty)$. An essential tool in this study is the $\epsilon$-regularity theorem, and given that for $p\neq 2$ such a theorem has been proved only if the target $N$ is a compact homogeneous space with a left invariant metric (see \cite{torowang}), we will restrict our study to this case. It is worth noticing that the $\epsilon$-regularity theorem is the only part where the homogeneity of $N$ plays a role, the rest of the arguments are valid for any compact target manifold. 

Using blow-ups and dimension reduction arguments, we will prove that the defect measure can be nonzero only if $p$ is an integer and if there exists a $C^{1,\alpha}$ stationary $p$-harmonic map from $S^p$ to $N$. Thus if we assume that $p$ is not an integer or that such a map doesn't exist, we recover all the regularity results proved in the minimizing case. In particular, we obtain

\begin{theorem}
 Let $u:B_2(0)\to N$ be a stationary $p$-harmonic map, where $N$ is a smooth compact homogeneous space with a left invariant metric. If $p$ is not an integer, then for all $\epsilon>0$:
\begin{gather*}
  \Vol\ton{B_r(\cS(u))\cap B_1(0)}\leq C(m,N,p,\epsilon) r^{\floor p +1-\epsilon}\, .
\end{gather*} 
Moreover, for all $p$ and under the additional assumption \eqref{B}, we can improve the previous estimate to
\begin{gather*}
  \Vol\ton{B_r(\cS(u))\cap B_1(0)}\leq C r^{b+2-\eta}\, .
 \end{gather*}  
\end{theorem}

As in the minimizing case, we also prove the following integrability results.
\begin{corollary}
Under the hypotheses of the previous theorem, for all $\epsilon>0$, $\nabla u\in L^{\floor{p}+1-\epsilon}(B_{1}(0))$ with
 \begin{gather*}
  \int_{B_{1}(0)} \abs{\nabla u}^{\floor{p}+1-\epsilon}\leq C(m,\Lambda,N,p,\epsilon)\, .
 \end{gather*}
Moreover, under the additional assumption \eqref{B}, $\nabla u\in L^{b+2-\epsilon}(B_{1}(0))$ with
 \begin{gather*}
  \int_{B_{1}(0)} \abs{\nabla u}^{b+2-\epsilon}\leq C(m,\Lambda,N,p,\epsilon)\, .
 \end{gather*}
\end{corollary}

Also the estimates for the borderline case carry over immediately.
\begin{theorem}
Under the hypothesis of the previous theorem, suppose that $p$ is not an integer and $m=\floor{p}+1$, or that, for any $p$, $m=b+2$ under the additional assumption \eqref{B}. Let $u$ be a stationary $p$-harmonic map $u:B_2(0)\to N$, where 
 \begin{gather*}
  \int_{B_2(0)}\abs{\nabla u}^p dV \leq \Lambda\, .
 \end{gather*}
Then
 \begin{gather*}
  \# \cS(u)\cap B_{1}(0)\leq C(p,\Lambda,N)\, .
 \end{gather*}
\end{theorem}


\begin{remark}
 For the sake of simplicity, we will only deal with the case $u:B_2(0)\subset \R^m \to N$. Given the local nature of the quantitative stratification, with simple modifications the results hold verbatim also for Riemannian manifolds with \eqref{eq_Mbounds}. The most important modifications needed for the general case will be pointed out in the study of $p$-minimizing maps (Section \ref{sec_minimizing}), while for $p$-stationary maps we refer to the analysis made by Lin for $p=2$, see \cite{lin_stat}, Section 5.
\end{remark}

\subsection{Sketch of the proof}
In this section, we will briefly sketch the main ideas involved in the quantitative stratification. 

It is well known that the monotonicity of the normalized energy $\theta_u(x,\cdot)$ implies the existence of (not necessarily unique) tangent maps for $u$ at every point (see for example \cite{simon}). Tangent maps are necessarily homogeneous weakly harmonic maps, and one says that a tangent map is $k$-symmetric if it is homogeneous and invariant wrt a $k$-dimensional subspace of $\R^m$ (for precise definitions, see Section \ref{sec_pre}). This allows to define a standard stratification of the domain of $u$ based on the number of symmetries of tangent maps. More precisely, for any integer $k\in[0,m]$ we define $\cS^k$ as the set of points $x$ such that all tangent maps at $x$ are not $k+1$ symmetric.

In a manner similar to \cite{chna1} and \cite{ChNa2}, we will define a quantitative stratification which refines the standard one. Roughly speaking, for fixed $r,\eta>0$ the quantitative stratification separates the points $x$ based on the number of $\eta$-{\it almost} symmetries of an approximate tangent map of $u$ at scales  $\geq r$; for a more precise statement see Definition \ref{deph_qstrat}.

The essential point of this paper is to prove Minkowski-type volume estimates on the quantitative strata, as opposed to the weaker Hausdorff estimates on the standard ones.

The key ideas involved in proving the estimates for the quantitative stratification are the {\it energy decomposition}, the $\epsilon$\textit{-regularity theorem} and {\it cone-splitting}.

In general, cone-splitting is the principle that, in the presence of conical structure, an object which is symmetric with respect to two distinct points automatically enjoys a higher order symmetry.

For example, in the setting of this article homogeneity with respect to a point plays the role of conical structure.
A function $h$ is said to be homogeneous wrt to a point, or equivalently $0$-symmetric at a point, if it is constant on the rays through that point. It is immediate to see that if $h$ is homogeneous with respect to two distinct points, then it is automatically constant on all lines parallel to the one joining these points.

In our terminology, we can rephrase this by saying that if a function is $0$-symmetric at two distinct points, then the function is actually $1$-symmetric. Using a simple compactness argument, it is possible to turn this statement into a quantitative cone-splitting for $p$-harmonic maps (see Proposition \ref{prop_conespl}). Roughly speaking, we will prove that if a function is \textit{almost} $0$-symmetric at two reasonably distant points, then it is actually \textit{almost} $1$-symmetric.

The $\epsilon$-regularity theorem provides a link between the strata $\cS^k$ and the singular set $\cS(u)$. Indeed, we will show that if a minimizing map $u$ is close enough in the appropriate sense to an $(m-\floor p)$-symmetric function, then $\nabla u$ is bounded, and $u$ does not have singular points in its domain. Equivalently, $\cS(u)\subset \cS^{m-\floor p-1}$.

The energy decomposition will exploit this by decomposing the space $B_1(0)$ based on which scales $u$ looks almost $0$-symmetric. On each such piece of the decomposition, nearby points automatically either force higher order symmetries or an improved covering of the space. By the $\epsilon$-regularity theorem, if a function has enough approximate symmetries then it is regular, and thus we obtain a good covering of the singular set in each piece of the decomposition. 
The final theorem is obtained by noting that, thanks to the monotonicity properties of the normalized energy, there are far fewer pieces to the decomposition than might  {\it apriori} seem possible.

The volume estimates on the singular points are an easy corollary of the estimates on the quantitative strata and a $\epsilon$-regularity type theorems from \cite{SU,pHL} for the minimizing case, and from \cite{beth,torowang} for the stationary one. Note that in the stationary case and for generic $p$, the $\epsilon$-regularity theorem has been proved only for homogeneous target manifolds. For this reason, we will restrict our study to this setting.

\paragraph{Regularity scale}
Actually the main estimates will not just be on $\cS(u)$, but on $\cB_r(u)$, an even bigger set. Indeed, we will be able to bound not only the size of the singular points, but also the size of the points where the gradient is big. Since the precise definition of $\cB_r(u)$ is rather technical (see \ref{deph_cbr}), here we only point out that
\begin{gather*}
 \cS(u)\subset \cB_r(u)\subset \cur{x\ \ s.t. \ \ \abs{\nabla u}(y)\leq r^{-1}  \ \ \forall y\in B_r(x)}^C\, .
\end{gather*}
Since the techniques described above are quantitative in nature, it should not be surprising that we are able to obtain these kind of quantitative results.

\vspace{3mm}

By using a refined covering, we will also improve the estimates in the case of isolated singularities and obtain a definite bound on the number of singular points.

\subsection{Preliminary properties}\label{sec_pre}
In this section we recall some of the basic properties related to normalized energy and homogeneous maps.
\begin{definition}
 For $u\in W^{1,p}(B_1(0),N)$, and for all $x,r$ such that $B_r(x)\subset B_1(0)$, define
 \begin{gather*}
  \theta_u(x,r) = r^{p-m} \int_{B_r(x)} \abs{\nabla u}^p dV\, .
 \end{gather*}
\end{definition}

A crucial property of stationary (and thus also of minimizing) $p$-harmonic maps is the monotonicity of $\theta(x,r)$ wrt $r$. The monotonicity follows from this well-known first variational formula (see for example \cite[eq. 1.3]{lin_stat}).
\begin{proposition}
 Let $u$ be a stationary $p$-harmonic map $u:B_1(0)\subset \R^m\to N$. Then for all smooth compactly supported vector fields $\xi\in C^\infty_c(B_1(0),\R^m)$,
 \begin{gather}
   \label{eq_p-stat} \int_{B_r(x)} \abs{\nabla u}^{p-2} \qua{\abs{\nabla u}^2 \delta^i_j - p \nabla^i u \nabla_j u }\partial_i \xi^j dV =0\, .
 \end{gather}
\end{proposition}

\begin{proposition}\label{prop_mono_stat}
Let $u$ be a stationary $p$-harmonic map $u:B_1(0)\subset \R^m\to N$, then the normalized $p$-energy is monotone nondecreasing in $r$. In particular for a.e. $r>0$:
\begin{gather}\label{eq_p-der}
 \frac{d}{dr} \theta(x,r) = pr^{p-m}\int_{\partial B_r(x)} \abs{\nabla u}^{p-2} \abs{\frac{\partial u}{\partial n}}^2 dS \geq 0\, .
\end{gather}
\end{proposition}

\begin{remark}
 If $u$ is defined on a Riemannian manifold, then $\theta(x,r)$ is not monotone but only ``almost'' monotone in the following sense: there exists a constant $C$ depending on $m$, $N$ and $p$ such that $e^{Cr}\theta(x,r)$ is monotone for all $r\leq \operatorname{inj}(M)$. See \cite[section 7]{pHL} for details in the minimizing case (the stationary case is completely analogous). This version of almost monotonicity is enough for all our purposes.
\end{remark}

As it is clear from equation \eqref{eq_p-der}, the normalized energy is very much related to homogeneous maps, of which we recall the definition here.
\begin{definition}
 We say that $h\in W^{1,p}(\R^m,N)$ is a homogeneous function of degree zero wrt the origin if for a.e. $\lambda >0$ and $x\in \R^m$:
 \begin{gather*}
  h(\lambda x) = h(x) \, ,
 \end{gather*}
or equivalently if $\frac{\partial h}{\partial n}=0$. We say that $h$ is a $k$-symmetric function if $h$ is homogeneous of degree zero and there exists a subspace $V$ of $\R^m$ of dimension $k$ such that
\begin{gather*}
 h(x+y)= h(x)
\end{gather*}
for a.e. $x\in \R^m$ and $y\in V$. 
\end{definition}

\begin{remark}
 For simplicity, from now on we will use the terms $0$-symmetric, homogeneous and homogeneous of degree zero as equivalent.
\end{remark}

Evidently, $h$ is $m$-symmetric if and only if it is a.e. constant.

\begin{remark}\label{rem_khom}
 By simple considerations, it is easy to see that the class of homogeneous functions $h:\R^m\to N$ is closed in the $L^p$ topology for any $p<\infty$. Moreover, if $h$ is homogeneous wrt the points $\{x_i\}$, then $h$ is symmetric wrt the affine space spanned by these points.
\end{remark}

We define also almost homogeneous functions according to their closeness to homogeneous functions. Before doing so, we define the blow-ups $T_{x,r}^u$.
\begin{definition}
 For $x\in B_1(0)$ and $r\leq 1$, define $T_{x,r}^u : B_1(0)\subset \R^m \to N$ by
 \begin{gather*}
T_{x,r}^u (y) \equiv u(x+r y)\, .
 \end{gather*}
For ease of notation we will write $T_{x,r}$ instead of $T_{x,r}^u$ when no ambiguity is possible.
 \end{definition}
 \begin{remark}In case $M$ is a Riemannian manifold, it is natural to replace the Euclidean blow-up with the one given by the exponential map. In particular, in this case we would define $T^u_{x,r}:B_1(0)\subset T_x(M)\to \R$ by
 \begin{gather*}
  T^u_{x,r} (y) \equiv u\ton{\exp_x(ry)}\, .
 \end{gather*}
\end{remark}
\begin{remark}[Scale invariance]
 From the definition of normalized energy, it is immediate to see that $\theta$ is scale-invariant. In other words
 \begin{gather*}
  \theta_u (x,r) = \theta_{T^u_{x,r}} (0,1)\, .
 \end{gather*}
\end{remark}

\begin{definition}
 We say that $u$ is $(k,\epsilon,r,x)$-symmetric if there exists a $k$-symmetric map $h$ such that
 \begin{gather*}
  \fint_{B_1(0)} d\ton{T^u_{x,r},h}^p dV <\epsilon\, .
 \end{gather*}
\end{definition}
With this definition, we can define the strata $\cS^{k}_{\eta,r}$ by:
\begin{definition}\label{deph_qstrat}
 Given a $p$-minimizing map $u$, an integer $k\geq 0$ and $r,\eta>0$, we define
 \begin{gather*}
  \cS^k_{\eta,r} = \cur{x\in B_1(0) \ \ s.t. \ \ \forall s\in [r,1]\, , \ \ u \ \ \text{is NOT} \  \ \ton{k+1,\eta,s,x}\text{-symmetric}}.
 \end{gather*}
\end{definition}

\section{Minimizing maps}\label{sec_minimizing}
The aim of this chapter is to prove the volume estimates on the strata $\cS^k_{\eta,r}$ for $p$-minimizers, and use them to prove regularity results. We start by proving a quantitative cone-splitting theorem (one could call it an ``almost'' cone-splitting).
\subsection{Cone-splitting theorem}
The cone-splitting theorem is the quantitative version of Remark \ref{rem_khom}. Using a simple compactness argument, we see that if $u$ is almost symmetric with respect to a set of points, and if this set of points ``almost spans'' a $k$ dimensional space, then $u$ is almost $k$ symmetric.

\begin{proposition}\label{prop_conespl}
 Let $u$ be a $p$-minimizing map with $\int_{B_2(0)}\abs{\nabla u}^p\leq \Lambda$, and fix some $\eta,\tau>0$. Then there exists $\epsilon=\epsilon(m,N,\Lambda,p,\eta,\tau)$ such that if
 \begin{enumerate}
  \item $u$ is NOT $(k+1,\eta,r,x)$-symmetric;
  \item $u$ is $(0,\epsilon,2r,x)$-symmetric;
 \end{enumerate}
then there exists a $k$-dimensional plane $V$ such that
\begin{gather*}
 \cur{y\ \ s.t. \ \ u\ \ \text{is} \ \ (0,\epsilon,2r,y)\text{-symmetric}}\cap B_r(x) \subset B_{\tau r} (V)\, ,
\end{gather*}
where $B_r(S)=\cur{x \ \ s.t. \ \ d(x,S)<r}$ is the tubular neighborhood of radius $r$ around the set $S$.
\end{proposition}
\begin{proof}
For convenience, we fix $x=0$ and $r=1$. Suppose by contradiction that the proposition is false. Then for each fixed $\eta$ and $\tau$, we can find a sequence of $p$ minimizing maps $u_i$ and a sequence of points $x_0^{(i)},\cdots,x_{k+1}^{(i)}\in B_1(0)$ such that
\begin{enumerate}
 \item $x_0=0$,
 \item $u_i$ is $(0,i^{-1},2,x_j^{(i)})$ symmetric for all $j$,
 \item for all $j=1,\cdots,k+1$, $d\ton{x_j^{(i)},\operatorname{span}\ton{x_0^{(i)},x_1^{(i)},\cdots,x_{j-1}^{(i)}}} \geq \tau$,
 \item $\int_{B_2(0)}\abs{\nabla u_i}^p\leq \Lambda$ .
\end{enumerate}
By compactness, $u_i$ (sub)converges weakly in the $W^{1,p}$ sense to a function $u$. According to \cite[Corollary 2.8]{pHL}, since $u_i$ are $p$-minimizers the convergence is also strong $W^{1,p}$ sense, and it is a minimizer by \cite{luck} (see also \cite[section 2.9]{simon}).

Moreover, by passing to a subsequence if necessary, we have $\lim_{i\to \infty}x^{(i)}_j = x_j$, and $\operatorname{span}(x_j)_{j=0}^{k+1}$ is a $k+1$ dimensional subspace.

The almost homogeneity properties of $u_i$ imply that $u$ is homogeneous with respect to all $x_j$ on $B_2(x_j)\supset B_1(0)$, and thus it is $k+1$ symmetric on $B_1(0)$. Since $u_i$ converges to $u$, for $i$ sufficiently large $u_i$ has to be $(k+1,\eta,1,0)$ symmetric, which is a contradiction.
\end{proof}

\subsection{Energy pinching and almost homogeneity}\label{sec_normen}
An immediate consequence of the monotonicity property (or better, of equation \eqref{eq_p-der}), is that if $\theta_u(x,r_1)=\theta_u(x,r_2)$, then $u$ is homogeneous wrt $x$ on the annulus $B_{r_2}(x)\setminus B_{r_1}(x)$. By a simple compactness argument, we can prove that if the normalized energy is sufficiently pinched, i.e. if $\theta(x,r)-\theta(x,\chi r)$ is small enough, then $u$ is almost homogeneous. This gives a very powerful characterization of almost homogeneous functions, specially given the monotonicity of $\theta$. Indeed, if we consider a sequence of scales $r_k=\chi^{-k}$, by monotonicity only for a bounded number of $k$ the difference $\theta(x,r_k)-\theta(x,r_{k+1})$ can be big. This proves that, for each $x$, $p$-minimizers are almost homogeneous wrt $x$ at all but a bounded number of scales.
\begin{theorem}\label{th_qrig}
 Let $u$ be a $p$-minimizer with $\int_{B_2(0)}\abs{\nabla u}^p dV \leq \Lambda$, $x\in B_1(0)$ and $r\leq 1$. Then for every $\epsilon>0$, there exists $\delta=\delta(m,N,\Lambda,p,\epsilon)$ and $0<\chi=\chi(m,N,\Lambda,p,\epsilon)\leq 1/2$ such that
 \begin{gather*}
 \theta(x,r)-\theta(x,\chi r)\leq \delta
 \end{gather*}
implies that $u$ is $(0,\epsilon,r,x)$-symmetric.
\end{theorem}
\begin{proof}
 Given the scale-invariant nature of this statement, we can assume without loss of generality that $x=0$ and $r=1$. Consider a sequence of $p$-minimizers $u_i$ with $\int_{B_2(0)}\abs{\nabla u_i}^p\leq \Lambda$ and 
 \begin{gather*}
  \theta_{u_i}(0,1)-\theta_{u_i}(0,i^{-1})\leq i^{-1}\, .
 \end{gather*}
By weak compactness, we can assume that $u_i$ (sub)converges weakly in $W^{1,p}(B_1(0))$ to some $u$.

In order to prove that $u$ is homogeneous, consider that $u_i$ are $p$-minimizers. Thus $u_i$ converge {\it strongly} to $u$, and in particular $\theta_0^{u}(r)$ is constant for $r\in (0,1)$. Thus $u$ is homogeneous on $B_1(0)$.

Alternatively, one can use an argument similar to the proof of \cite[Lemma 2.5]{SU} to prove the homogeneity of the tangent map.
\end{proof}
In case of a Riemannian manifold, the previous statement needs to be tweaked a little. Indeed, the limit function $u$ in the previous proof is defined on $B_1(0)\subset T_x(M)$ and it minimizes the $p$-energy with respect to the metric on the manifold, not with respect to the standard Euclidean metric. Moreover, since $\theta$ in this case is only almost monotone, $u$ need not be homogeneous. For these reasons, we also need $r$ in the previous theorem to be effectively small, so that the geodesic ball $B_r(0)$ is close enough to the Euclidean ball with the same radius.
\begin{theorem}\label{th_qrig_riemm}
 Let $u:B_2(0)\subset M \to N$ be a $p$-minimizer with $\int_{B_2(0)}\abs{\nabla u}^p dV \leq \Lambda$, $x\in B_1(0)$ and $r\leq 1$. Then for every $\epsilon>0$, there exists $\delta=\delta(m,N,\Lambda,p,\epsilon)$, $r_0=r_0(m,N,\Lambda,p,\epsilon)$ and $0<\chi=\chi(m,N,\Lambda,p,\epsilon)\leq 1/2$ such that $r\leq r_0$ and
 \begin{gather*}
 \theta(x,r)-\theta(x,\chi r)\leq \delta
 \end{gather*}
implies that $u$ is $(0,\epsilon,r,x)$-symmetric.
\end{theorem}
\begin{proof}
The proof proceeds as in the Euclidean case. In particular, by contradiction we build a sequence $u_i$ which minimize the Riemannian $p$-energy on $B_{i^{-1}}(0)$. By the almost monotonicity of $\theta$, and by the assumptions \eqref{eq_Mbounds}, the sequence $T_i=T_{0,i^{-1}}^{u_i}$ has a uniform $W^{1,p}(B_1(0))$ bound. Thus $T_i$ has a weakly convergent subsequence.

The strong convergence of $T_i$ and the fact that $T$ is a Euclidean $p$-minimizer can be proved by a simple adaptation of \cite[Proposition 4.7 and Proposition 5.2]{SU}. Alternatively, one can use the technique of $\epsilon$-almost minimizers developed in \cite{luck} (see also \cite[section 2]{simon}).
\end{proof}
\begin{remark}
 Since $r_0$ depends only on $m,N,\Lambda,p$, the extra assumption $r\leq r_0$ does not change in a significant way any of the volume estimates we want to prove.
\end{remark}

\subsection[e-regularity theorem]{$\epsilon$-regularity theorem}
The last important ingredient needed for the proof of our main theorems is the so-called $\epsilon$-regularity theorem for $p$-minimizers. This theorem states that if $u$ is close enough to a constant in the $L^p$ sense, then $u$ is regular. More precisely we have
\begin{theorem}[$\epsilon$-regularity theorem]\cite[Corollary 2.7, Theorem 3.1]{pHL}\label{th_eps_pHL}
 Let $u$ be a $p$-minimizing map $u:B_2(0)\to N$. Then for every $\Lambda>0$, there exists constants $\delta(\Lambda,m,N,p)>0$, $\alpha(m,N,p)>0$ and $C(m,N,p)>0$ such that if
 \begin{gather*}
  \int_{B_2(0)} \abs {\nabla u} ^p dV \leq \Lambda \quad \text{and} \quad \int_{B_2(0)} d(u,w)^p dV \leq \delta\, ,
 \end{gather*}
where $w$ is any fixed point $w\in N$, then $f$ is $(1,\alpha)$-Holder continuous on $B_{1}(0)$ and
\begin{gather*}
 \norm{f}_{C^{1,\alpha}(B_{1}(0))}\leq C\, .
\end{gather*}
\end{theorem}

The authors is \cite{pHL} use the $\epsilon$-regularity theorem and the monotonicity of $\theta$ to prove that the Hausdorff dimension of $\cS(u)$ is bounded above by $m-\floor{p}-1$. In particular, this implies that all $m-\floor{p}$ symmetric $p$-minimizers are constant. Using this and a simple compactness argument, we can improve the $\epsilon$-regularity theorem to the following version.
\begin{theorem}\label{th_eps_+}
 Let $u$ be a $p$-minimizing map $u:B_2(0)\to N$ with $\int_{B_2(0)} \abs {\nabla u} ^p dV \leq \Lambda$. There exists constants $\epsilon(\Lambda,m,N,p)>0$ and $\alpha(m,N,p)>0$ such that if $u$ is $(m-\floor p,\epsilon,1,0)$-symmetric then $u$ is $(1,\alpha)$-Holder continuous on $B_{1}(0)$ and
\begin{gather*}
 \norm{u}_{C^{1,\alpha}(B_{1}(0))}\leq 1\, .
\end{gather*}
\end{theorem}
\begin{proof}
This theorem follows from the previous one and an easy compactness argument.


Suppose by contradiction that this theorem is false. Then there exists a sequence of $p$-minimizing maps $u_i$ and a sequence of $m-\floor{p}$ symmetric maps $h_i$ such that
 \begin{gather*}
  \int_{B_1(0)} \abs {\nabla u_i} ^p dV \leq \Lambda \quad \text{and} \quad \int_{B_1(0)} d(u_i,h_i)^p dV \leq i^{-i}\, ,
 \end{gather*}
 but for which $\int_{B_1(0)}d(u_i,w)^p dV \geq \epsilon$ for all $w\in N$.

Given the compactness of $N$, $h_i$ has a subsequence which converges strongly in $L^p(B_1(0))$ to an $m-\floor{p}$ symmetric function $h$. Moreover, $u_i$ has a subsequence which converges strongly in $W^{1,p}(B_1(0))$ to a $p$-minimizer $u$.

Thus $h=u$ is an $m-\floor{p}$ symmetric $p$-minimizer, which is necessarily constant by \cite{pHL}. The previous theorem then ensures that $u_i$ converges to $h$ also in the sense of $C^{1,\alpha/2}$, and this concludes the proof.
\end{proof}


Under the additional assumptions \eqref{A}, we can improve the previous results and show that any almost $m-a-1$ symmetric map is constant. 
\begin{corollary}\label{cor_eps_+}
 Suppose that condition \eqref{A} holds, and let $u$ be a $p$-minimizing map $u:B_2(0)\to N$ with $\int_{B_2(0)} \abs {\nabla u} ^p dV \leq \Lambda$. There exists constants $\epsilon(\Lambda,m,N,p)>0$ and $\alpha(m,N,p)>0$ such that if $u$ is $(m-a-1,\epsilon,1,0)$-symmetric then $u$ is $(1,\alpha)$-Holder continuous on $B_{1}(0)$ and
\begin{gather*}
 \norm{\nabla u}_{C^{1,\alpha}(B_{1}(0))}\leq 1\, .
\end{gather*}
\end{corollary}
\begin{proof}
A key element in the proof of the previous theorem is that all minimizing $p$-harmonic maps which are $(m-\floor p)$-symmetric are necessarily constant. In the next lemma, we show using a standard argument that under assumption \eqref{A} any $(m-a-1)$-symmetric minimizing map is constant. The rest of the proof carries over immediately.
\end{proof}

\begin{lemma}
 Under the additional assumptions \eqref{A}, all $(m-a-1)$-symmetric $p$-minimizing maps $h:\R^m\to N$ are constant. 
\end{lemma}
\begin{proof}
 Suppose by contradiction that there exists such a map $h$ with a singularity, and let $S$ be its invariant subspace of dimension $\geq m-a-1$. By invariance, the map $h$ induces a minimizing map $\tilde h:\R^{a+1}\to N$. If the origin is the only isolated singularity of $\tilde h$, then it is immediate to obtain a continuous $p$-minimizing map $\hat h:S^{a}\to N$, which is trivial by assumption, thus $h$ would be constant.
 
 We finish the proof by induction. If $h$ has a singularity at $x\not \in S$, then by the $\epsilon$-regularity theorem $\theta_h(x,0)>\epsilon$. Let $h'$ be a tangent map at $x$, thus $h'$ is a nonconstant minimizing map which is easily seen to be invariant both with respect to $S$ and with respect to the subspace generated by $x$. In other words $h'$ is $(m-a)$-symmetric. By the previous argument, $h'$ induces a minimizing map from $\R^a$ to $N$, and this map cannot have an isolated singularity at the origin. If this map had other singularities, by induction we would obtain a minimizing map with one more symmetry. Since $m-\floor p$ symmetric maps are necessarily constant, the proof is finished.
\end{proof}

\subsection{Regularity scale}\label{sec_regscale}
Given the scale-invariant properties of the problem we are focusing on, it is convenient to define some scale-invariant quantities measuring the regularity of the function $u$.
\begin{definition}\label{deph_regscale}
 Let $\alpha=\alpha(m,N,p)>0$ be the one given by Theorem \ref{th_eps_+}. We define the scale-invariant norm $\norm{u}_{x,r}$ of $u$ at the point $x$ at scale $r$ as
 \begin{gather*}
  \norm{u}_{x,r} = \begin{cases} r\sup_{y\in B_r(x)}\cur{\abs {\nabla u(y)} } + r^{1+\alpha}\sup_{z\neq y \in B_r(x)}\cur{\frac{\abs {\nabla u(y)-\nabla u(z)}}{\abs{y-z}^\alpha} },&if\  u\in C^{1,\alpha}(B_r(x))\\
                    +\infty,& otherwise.
                   \end{cases}
 \end{gather*}
 We define also the regularity scale by
 \begin{gather*}
  r_u(x)=\sup_{r\geq 0}\cur{\norm{u}_{x,r} \leq 1}\, .
 \end{gather*}

\end{definition}

\begin{remark}
Note that this definition is scale-invariant, in the sense that $\norm{T^u_{x,r}}_{0,1} = \norm{u}_{x,r}$.  Moreover $\norm{\cdot}_{x,r}$ is monotone in $r$. In particular, if $r\leq s$, then
 \begin{gather*}
  \norm u _{x,r} \leq \norm u_{x,s}\, .
 \end{gather*}
\end{remark}

\begin{definition}\label{deph_cbr}
 Let $u$ be a $p$-minimizing map as in the statement of Theorem \ref{th_main_proof}, and $r>0$. Define the set
 \begin{gather*}
  \cB_r(u)= \cur{x\in B_1(0) \ \ s.t. \ \ \norm{T_{x,r} }_{0,1} = \norm u _{x,r}> 1}=\cur{x\in B_1(0) \ \ s.t. \ \ r_u(x)< r}\, .
 \end{gather*}
\end{definition}

We can restate the $\epsilon$-regularity theorem in the following form.
\begin{theorem}\label{th_eps_hom}
Let $u$ be a $p$-minimizing map $u:B_2(0)\to N$, where $N$ is compact (without boundary) and
 \begin{gather*}
  \int_{B_2(0)}\abs{\nabla u}^p dV \leq \Lambda\, .
 \end{gather*}  Then there exists a positive $\epsilon=\epsilon(m,N,\Lambda,p)$ such that, for all $r\leq 1$,
 \begin{gather*}
  \cS(u)\cap B_1 (0) \subset \cB_r(u) \cap B_1 (0) \subset \cS^{m-\floor{p}-1}_{\epsilon,r} (u)\cap B_1 (0) \, .
 \end{gather*}
Under the additional assumption \eqref{A}, we can improve the previous result to
 \begin{gather*}
  \cS(u)\cap B_1 (0) \subset \cB_r(u) \cap B_1 (0) \subset \cS^{m-a-2}_{\epsilon,r} (u)\cap B_1 (0) \, .
 \end{gather*}

\end{theorem}
\begin{proof}
 The inclusion $\cS(u)\cap B_1 (0) \subset \cB_r(u)\cap B_1 (0) $ is immediate, while the inclusion
 \begin{gather*}
  \qua{\cS^{m-\floor{p}-1}_{\epsilon,r} (u) \cap B_1 (0) }^C \subset \qua{\cB_r(u) \cap B_1 (0) }^C
 \end{gather*}
is just a scale-invariant form of Theorem \ref{th_eps_+} and Corollary \ref{cor_eps_+}.
\end{proof}

\subsection{Volume estimates on the strata}
By applying the quantitative stratification technique (see \cite{ChNa2}), we now prove effective volume bounds on the singular strata $\cS_{\eta,r}$. In the next section, we will see how these bounds imply effective regularity estimates on the map $u$.
\begin{theorem}\label{th_main_proof}
Let $u$ be a $p$-minimizing map as in the statement of Theorem \ref{th_eps_hom}. 
Then for every $\eta>0$, there exists a constant $C(m,N,\Lambda,p,\eta)$ such that
\begin{gather*}
 \Vol\ton{B_r(\cS_{\eta,r}^k(u))\cap B_{1}(0)}\leq C r^{m-k-\eta}\, .
\end{gather*}
\end{theorem}

The scheme of the proof is the following: fix $\gamma=c_0^{-2/\eta}>0$, where $c_0=c_0(m)$ is the dimensional geometrical constant appearing in the proof of Lemma \ref{lemma_cov2}. Up to increase the value of $c_0$, we can suppose that $\gamma<1/10$.
We will prove that there exists a covering of $\cS^{k}_{\eta,\gamma^j}$ made of nonempty open sets in the collection $\{\cC^k_{\eta,\gamma^j}\}$. Each set $\cC^k_{\eta,\gamma^j}$ is the union of a controlled number of balls of radius $\gamma^j$.

This will give the desired volume bound. In particular:

\begin{lemma}[Decomposition Lemma]\label{lemma_dec}
There exists $c_0(m),c_1(m)>0$ and $D(m,N,\gamma,\Lambda,p,\eta)>1$ such that for every $j\in \N$:
\begin{enumerate}
 \item $\cS^k_{\eta,\gamma^j}\cap B_{1}(0)$ is contained in the union of at most $j^D$ \textit{nonempty} open sets $\cC^k_{\eta,\gamma^j}$
 \item Each $\cC^k_{\eta,\gamma^j}$ is the union of at most $(c_1\gamma ^{-m})^D (c_0\gamma^{-k})^{j-D}$ balls of radius $\gamma^j$
\end{enumerate}
\end{lemma}
Once this lemma is proved, Theorem \ref{th_main_proof} follows easily.
\begin{proof}[Proof of Theorem \ref{th_main_proof}]
Since we have a covering of $\cS^k_{\eta,\gamma^j}\cap B_{1}(0)$ by balls of radius $\gamma^j$, it is easy to get a covering of $B_{\gamma^j}\ton{\cS^k_{\eta,\gamma^j}}\cap B_{1}(0)$. In fact it is sufficient to double the radius of the original balls. Now it is evident that:
\begin{gather*}
 \Vol\qua{B_{\gamma^j}\ton{\cS^k_{\eta,\gamma^j}}\cap B_{1}(0)} \leq j^D \ton{(c_1\gamma^{-m})^D (c_0\gamma^{-k})^{j-D}} \omega_m 2^m \ton{\gamma^j}^m
\end{gather*}
where $\omega_m$ is the volume of the $m$-dimensional unit ball. By plugging in the simple rough estimates
\begin{gather*}
 j^D \leq c(m,N,\Lambda,p,\eta)\ton{\gamma^j}^{-\eta/2}\, ,\\
\notag (c_1\gamma^{-m})^D(c_0\gamma^{-k})^{-D}\leq c(m,N,\Lambda,p,\eta)\, ,
\end{gather*}
and using the definition of $\gamma$, we obtain the desired result.
\end{proof}

\paragraph{Proof of the Decomposition Lemma}
Now we turn to the proof of the Decomposition Lemma. In order to do this, we define a new quantity which measures the non-homogeneity of $u$ at a certain scale.

Given $u$ as in Theorem \ref{th_main_proof} and $\epsilon>0$, we divide the set $B_{1}(0)$ into two subsets according to the behaviour of the points with respect to their quantitative symmetry. In particular, define
\begin{align*}
L_{r,\epsilon}(u)=&\{x\in B_{1}(0) \ s.t.\ u \ \ is \ \ (0,\epsilon,{r}/\ton{5\gamma},x)\text{-symmetric}\}\, ,\\
H_{r,\epsilon}(u)=&L_{r,\epsilon}(u)^C\, .
\end{align*}
Next, to each point $x\in B_{1}(0)$ we associate a $j$-tuple $T^j(x)$ of numbers $\{0,1\}$ in such a way that the $i$-th entry of $T^j(x)$ (which will be denoted by $T^j_i(x)$) is $1$ if $x\in H_{\gamma^i, \epsilon}(u)$, and zero otherwise. Then, for each fixed $j$-tuple $\bar T^j$, set:
\begin{gather*}
 E(\bar T^j) = \{x\in B_{1}(0) \ \ s.t. \ \ T^j(x)=\bar T^j\}
\end{gather*}
Also, we denote by $T^{ j-1}$, the $(j - 1)$-tuple obtained from $T^ j$ by dropping the last entry, and set $\abs{T^j}$ to be number of $1$ in the $j$-tuple $T^j$, i.e., $\abs{T^j(x)} = \sum_{i=1}^j T^j_i(x)$.

We will build the families $\{\cC^k_{\eta,\gamma^a}\}$ by induction on $a=0,\cdots,j$ in the following way.
For $a=0$, $\{\cC^k_{\eta,\gamma^0}\}$ consists of the single ball $B_{1}(0)$.

\paragraph{Induction step}
For fixed $a\leq j$, suppose that by induction we have already built the family $\cur{\cC^k_{\eta,\gamma^{a-1}}}$, and consider all the $2^a$ $a$-tuples $\bar T^a$. Label the sets of balls in the family $\{\cC^k_{\eta,\gamma^a}\}$ by all the possible $a$-tuple $\bar T^a$. We will build $\cC^k_{\eta,\gamma^a}(\bar T^a)$ inductively as follows.  For each ball $B_{\gamma^{a-1}}(y)$ in $\{\cC^k_{\eta,\gamma^{a-1}}(\bar T^{a-1})\}$ take a minimal covering of $B_{\gamma^{a-1}}(y)\cap \cS^{k}_{\eta,\gamma^j} \cap E(\bar T^a)$ by balls of radius $\gamma^a$ centered at points in $B_{\gamma^{a-1}}(x)\cap \cS^k_{\eta,\gamma^j}\cap E(\bar T^a)$. Note that it is possible that for some $a$-tuple $\bar T^a$, the set $E(\bar T^a)$ is empty, and in this case $\{\cC^k_{\eta,\gamma^{a}}(\bar T^{a})\}$ is the empty set.

Now we need to prove that the minimal covering satisfies points 1 and 2 in Lemma \ref{lemma_dec}. We will do this in the next three lemmas.
\begin{remark}
For the moment let $\epsilon>0$ be an arbitrary fixed small quantity. Its value will be chosen in order to apply Proposition \ref{prop_conespl} with $\eta$ as in the statement of Theorem \ref{th_main_proof} and $\tau=10^{-1}\gamma$.
\end{remark}

\paragraph{Point 1 in Lemma}
As we will see below, we can use the monotonicity of $\theta$ to prove that for every $\bar T^j$, $E(\bar T^j)$ is empty if $\abs{\bar T^j}\geq D$. Since for every $j$ there are at most $\binom j D\leq j^D$ choices of $j$-tuples which do not satisfy such a property, the first point will be proved.

\begin{lemma}\label{lemma_K}
 There exists  $D=D(\epsilon,\gamma,m,N,\Lambda,p)$ such that $E(\bar T^j)$ is empty if $\abs{\bar T^j}\geq D$.
\end{lemma}

\begin{proof}
For $s<r$, we set
\begin{gather*}
 \cW_{s,r}(x)=\theta(x,r)-\theta (x,s)\geq 0\, .
\end{gather*}
If $(s_i,r_i)$ are \textit{disjoint} intervals with $\max\{r_i\}\leq 1/3$, then by monotonicity of $\theta$
\begin{gather}\label{eq_sum}
 \sum_i \cW_{s_i,r_i}(x)\leq \theta(x,1/3) -\theta(x,0) \leq C(m,p,\Lambda)\, .
\end{gather}

Let $\chi=\chi(m,N,\Lambda,p,\epsilon)$ be given by Theorem \ref{th_qrig} and let $A\in \N$ be such that $\gamma^{A} \leq \chi$.
Consider intervals of the form $(\gamma^{i-1}/5,  \gamma^{i+A-1}/5)$ for $i=1,2,...\infty$. By Theorem \ref{th_qrig}, there exists a $\delta$ independent of $x$ such that
\begin{gather*}
 \cW_{\gamma^{i-1}/5,\gamma^{i+A-1}/5}(x)\leq \delta \implies u \text{ is }(0,\epsilon,\frac{\gamma^{i-1}}{5},x)\text{-symmetric}\, .
\end{gather*}
in particular $x\in L_{ \gamma^{i},\epsilon}$, so that, if $i\leq j$, the $i$-th entry of $T^j$ is necessarily zero. By equation \eqref{eq_sum}, there can be only a finite number of $i$'s such that $\cW_{\gamma^{i-1}/5, \gamma^{i+A-1}/5}(x)>\delta$, and this number $D$ is bounded by:
\begin{gather}\label{eq_estK}
 D\leq A\frac{C(m,p,\Lambda)}{\delta}\, .
\end{gather}
This completes the proof.
\end{proof}

\paragraph{Point 2 in Lemma}
The proof of the second point in Lemma \ref{lemma_dec} is mainly based on Proposition \ref{prop_conespl}. In particular, for fixed $k$ and $\eta$ in the definition of $\cS^k_{\eta,\gamma^j}$, $\epsilon$ is chosen in such a way that Proposition \ref{prop_conespl} can be applied with $\tau = 10 ^{-1}\gamma$. Then we can restate the lemma as follows:
\begin{lemma}\label{lemma_cov}
 Let $\bar T^j_a =0$. Then the set $G=\cS^{k}_{\eta,\gamma^j}\cap B_{\gamma^{a-1}}(x)\cap E(\bar T^j)$ can be covered by $c_0(m)\gamma^{-k}$ balls centered in $G$ of radius $\gamma^{a}$.
\end{lemma}
\begin{proof}
 First of all, note that since $\bar T^j_a =0$, all the points in $E(\bar T^j)$ are in $L_{\epsilon, \gamma^a}(u)$.

The set $G$ is contained in $B_{10^{-1}\gamma^a}(V^k)\cap B_{\gamma^{a-1}}(x)$ for some $k$-dimensional subspace $V^k$.
Indeed, if there were a point $x\in G$, such that $x\not\in B_{10^{-1} \gamma^a}(V^k)\cap B_{\gamma^{a-1}}(x)$,
then by Proposition \ref{prop_conespl} (applied with $\tau=10^{-1}\gamma$ and $r=10^{-1}\gamma^{a-1}$) the map $u$ would be $(k+1,\eta,10^{-1}\gamma^{a-1},x)$-symmetric.
Since $10^{-1}\gamma^{a-1}>\gamma^j$, this contradicts $x\in \cS^k_{\eta,\gamma^j}$.
It is standard geometry that $V^k \cap B_{\gamma^{a-1}}(x)$ can be covered by $c_0(m)\gamma^{-k}$ balls of radius $\frac 9 {10} \gamma^a$,
and by the triangle inequality it is evident that the same balls with radius $\gamma^a$ cover the whole set $G$.
\end{proof}

If instead $\bar T^j_a =1$, then without any effort we can say that $G=\cS^{k}_{\eta,\gamma^j}\cap B_{\gamma^{a-1}}(x)\cap E(\bar T^j)$ can be covered by $c_1(m)\gamma^{-m}$ balls of radius $\gamma^a$. Now by a simple induction argument the proof is complete.
\begin{lemma}\label{lemma_cov2}
 Each (nonempty) $\cC^k_{\eta,\gamma^j}$ is the union of at most $(c_1\gamma ^{-m})^D (c_0\gamma^{-k})^{j-D}$ balls of radius $\gamma^j$.
\end{lemma}
\begin{proof}
 Fix a sequence $\bar T^j$ and consider the set $\cC^k_{\eta,\gamma^j}(\bar T^j)$. By Lemma \ref{lemma_K}, we can assume that $\abs {\bar T^j}\leq D$, otherwise there is nothing to prove since $\cC^k_{\eta,\gamma^j}(\bar T^j)$ would be empty.

Consider that for each step $a$, if $\bar T^j_a=0$, in order to get a (minimal) covering of $B_{\gamma^{a-1}}(x)\cap \cS^{k}_{\eta,\gamma^i}\cap E(\bar T^j) $ for $B_{\gamma^{a-1}}(x)\in \cC^k_{\eta,\gamma^{a-1}}(\bar T^j)$, we require at most $(c_0 \gamma^{-k})$ balls of radius $\gamma^{a}$. If $\bar T^j_a=1$, we need $(c_1\gamma^{m})$ balls. Since the latter situation can occur at most $D$ times, the proof is complete.
\end{proof}

\subsection{Regularity estimates}
In this section, we collect the main theorems for minimizing maps. As anticipated in the introduction, we obtain estimates not only on the singular set, but also quantitative estimates on the regularity scale and, as a corollary, sharp integrability conditions for the minimizers $u$.

First of all, we stress that the regularity properties of the minimizers strongly depend on $p$ and $m$. For example, it is well known that if $m\leq \floor{p}$, then all $p$-minimizers have no singular points, and thus are $C^{1,\alpha}$ functions. Moreover, as shown in \cite[section 3.6]{simon}, one can prove that there exist uniform $C^{1,\alpha}$ bounds on $u$ depending only on $m,\Lambda,N$.. In the following theorem, we give a short proof of this statement.
\begin{theorem}\label{th_m<p}
 If $m\leq \floor{p}$, then there exists a constant $C(p,\Lambda,N)$ such that
 \begin{gather*}
  \int_{B_2(0)} \abs{\nabla u}^p\leq \Lambda \quad \Longrightarrow \quad \norm{\nabla u}_{C^{0,\alpha}(B_{1}(0))}\leq C\, .
 \end{gather*}
 Under the additional assumption \eqref{A}, if $m\leq a+1$, then there exists a constant $C(p,\Lambda,N)$ and an exponent $\alpha(m,N,p)>0$ such that
 \begin{gather*}
  \int_{B_2(0)} \abs{\nabla u}^p\leq \Lambda \quad \Longrightarrow \quad \norm{\nabla u}_{C^{0,\alpha}(B_{1}(0))}\leq C\, .
 \end{gather*}
\end{theorem}
\begin{proof}
 By Theorem \ref{th_eps_+}, there exist $\epsilon$ and $\alpha$ such that if $u$ is $(0,\epsilon,r,x)$ symmetric, then $\norm{u}_{x,r} \leq 1$.
By Theorem \ref{th_qrig}, we can rephrase this last property as follows: there exist $\delta,\chi>0$ such that 
\begin{gather*}
 \theta(x,r)-\theta(x,\chi r)\leq \delta \quad \Longrightarrow \quad  \norm{u}_{x,r}\leq 1\, .
\end{gather*}
Now we argue in a way similar to the proof of Lemma \ref{lemma_K}. Consider the sequence of scales $r_k = \chi^k$. By monotonicity, there exists a $K(p,\Lambda,N)<\infty$ such that $\theta(x,r_k)-\theta(x,r_{k+1})\leq \delta$ for some $0\leq k \leq K$. This implies that $\norm{u}_{x,r_K}\leq 1$, and thus we obtain the desired bounds.

Using Corollary \ref{cor_eps_+} instead of Theorem \ref{th_eps_+}, we prove the second statement.
\end{proof}

Naturally, the interesting case is when $m\geq \floor p +1$. As a corollary of the estimates obtained in the previous section, we can prove the main theorem.
\begin{theorem}
Let $u:B_2(0)\to N$ be a minimizing $p$-harmonic map with $\int_{B_2(0)}\abs{\nabla u}^p dV\leq \Lambda$. For every $\eta>0$, there exists a constant $C(m,\Lambda,N,p,\eta)$ such that
\begin{gather}\label{eq_rj}
  \Vol\ton{B_r(\cS(u))\cap B_1(0)}\leq \Vol\ton{B_r(\cB_r(u))\cap B_1(0)}\leq C r^{\floor{p}+1-\eta}\, .
 \end{gather}
Moreover, under the additional assumption \eqref{A}, we can improve the previous estimate to
\begin{gather*}
  \Vol\ton{B_r(\cS(u))\cap B_1(0)}\leq \Vol\ton{B_r(\cB_r(u))\cap B_1(0)}\leq C r^{a+2-\eta}\, .
 \end{gather*} 
\end{theorem}
This theorem is a corollary of the estimates in Theorem \ref{th_eps_hom} and Theorem \ref{th_main_proof}. With this estimate, it is immediate to prove the following sharp integrability theorem.
\begin{corollary}
Let $u:B_2(0)\to N$ be a minimizing $p$-harmonic map with $\int_{B_2(0)}\abs{\nabla u}^p dV\leq \Lambda$. For all $\epsilon>0$, $\nabla u\in L^{\floor{p}+1-\epsilon}(B_1(0))$ with
 \begin{gather*}
  \int_{B_1(0)} \abs{\nabla u}^{\floor{p}+1-\epsilon}\leq C(m,\Lambda,N,p,\epsilon)\, .
 \end{gather*}
Moreover, under the additional assumption \eqref{A}, $\nabla u\in L^{a+2-\epsilon}(B_1(0))$ with
 \begin{gather*}
  \int_{B_1(0)} \abs{\nabla u}^{a+2-\epsilon}\leq C(m,\Lambda,N,p,\epsilon)\, .
 \end{gather*}
\end{corollary}
\begin{proof}
 The proof is an immediate corollary of the regularity scale estimates. Indeed, let $\eta=\epsilon/2>0$ and consider that for all $r>0$ we have by \eqref{eq_rj}
 \begin{gather*}
  \Vol\ton{\cur{x \ \ s.t. \ \ \abs{\nabla u(x)}>r^{-1}}}\leq C r^{\floor p +1-\epsilon/2} =C r^{\floor p +1-\epsilon} r^{\epsilon/2}\, .
 \end{gather*}
This immediately gives the desired integral estimates on $\abs{\nabla u}$.
\end{proof}

Note that the integrability is sharp. Indeed, consider the map $u:B_1(0)\subset \R^m\to S^{m-1}$ defined by $u(x)=x/\abs x$. This map is $p$-harmonic if $m>p$, but $\vert \nabla u \vert\not \in L^m(B_1(0))$. Thus we cannot improve the integrability to $\floor{p}+1$.

\subsection[Improved regularity for m equal to p+1]{Improved regularity for $m=\floor{p}+1$}\label{sec_isol_sing}
In this section, we focus on the special case $m=\floor p +1$ (or $m=a+2$ under the additional assumptions \eqref{A}). In this situation, it is known that singular points are isolated (see for example \cite{pHL}). We improve this result to an effective finiteness, which is new even in the case $p=2$. The next lemma describes the property that makes this case special.
\begin{lemma}
 Let $m=\floor p +1$, or let $m=a+2$ under the additional assumption \eqref{A}. There exists $\eta(p,\Lambda,N)$ such that if $u$ is $(0,\eta,r,x)$-symmetric and $(2r)^{2-m}\int_{B_{2r}(x)} \abs{\nabla u}^p\leq \Lambda $, then $u$ does not have singular points in the annulus $A_{r}(x) = B_{\frac{r}{2}}(x)\setminus B_{\frac{r}{4}}(x)$.
\end{lemma}
\begin{proof}
 We will only deal with the case $m=\floor p +1$, the other being completely analogous. Moreover, by scale and translation invariance, we can assume that $x=0$ and $r=1$.
 
 Consider by contradiction a sequence of minimizers $u_i$ which are $(0,\eta,1,0)$-symmetric and such that $\int_{B_2 (0)} \abs{\nabla u}^2\leq 2^{n-2}\Lambda$, and let $x_i$ be a singular point of $u_i$ inside the annulus $B_{1}(0)\setminus B_{1/2}(0)$. By passing to a subsequence, we can assume that $u_i\to u$ in the strong $W^{1,p}$ sense and that $x_i\to x$, where $x$ is a singular point for $u$. Since $u$ is a homogeneous minimizing map, and since $m=\floor p +1$, it cannot have a singular point away from the origin.
%
%
\end{proof}


As an immediate corollary, we can prove that all points in $\cS(u)$ are isolated.
\begin{lemma}
Under the hypothesis of the previous lemma, the singular points of $u$ are locally finite.
\end{lemma}
\begin{proof}
 Given the monotonicity of $\theta$, for each $x$ there exists an $r_x$ such that
 \begin{gather*}
  \theta(x,r_x)-\theta(x,0)\leq \delta\, .
 \end{gather*}
Then by applying the previous lemma to all $r\leq r_x$, we obtain that $u$ is continuous on $B_{r_x}(x)\setminus\cur{x}$. This proves that the $\cS(u)$ is an isolated close set, thus locally finite.
\end{proof}
By refining this lemma, we prove a uniform upper bound on the number of singular points.
\begin{theorem}
Suppose that $m=\floor{p}+1$, or that $m=a+2$ under the additional assumption \eqref{A}. Let $u$ be a $p$-minimizing map $u:B_1(0)\to N$, where 
 \begin{gather*}
  \int_{B_1(0)}\abs{\nabla u}^p dV \leq \Lambda\, .
 \end{gather*}
Then
 \begin{gather*}
  \# \cS(u)\cap B_{1}(0)\leq C(p,\Lambda,N)\, .
 \end{gather*}
\end{theorem}
\begin{proof}
Consider the sequence of scales $r_k=2^{-k}$. By an argument similar to the one in Lemma \ref{lemma_K}, for each fixed $x$ there exists at most $C(p,\Lambda,N)$ ``bad scales'', i.e., scales for which $u$ is not continuous on $A_{r_k}(x)$.

For any fixed $u$, the number of singular points in $B_{1}(0)$ is finite by the previous lemma. Let $S_0$ be this number, we will  prove by induction a uniform upper bound on $S_0$.

\paragraph{Induction step}
Define $T_i$ to be an infinite vector of zeros and ones, and let $\abs T = \sum_{i=1}^\infty T(i)$.

For $i=1$, consider all the balls of radius $2^{-1}$ centered at $x\in \cS(u)\cap B_1(0)$, and refine this covering of $\cS(u)$ by considering only a maximal subcovering such that $B_{2^{-2}}(x_j)$ are disjoint. 
By simple volume estimates the number of balls in this covering is at most $c(m)$. 

Consider a ball in this covering that contains the largest number of singular points, say $B_{2^{-1}}(x_1)$, containing $S_1$ singular points. If $S_1 = S_0\equiv \#\cS(u)$, equivalently if $B_{2^{-1}}(x_1)$ contains all the singular points, then set $T_1=0$, otherwise set $T_1=1$. In this second case,
\begin{gather*}
S_0> S_1 \geq c(m)^{-1}S_0\, .
\end{gather*}
Moreover there exists $y_1\in \cS(u)\cap B_1(0)\setminus B_{2^{-1}}(x_1)$. Thus for each $z\in B_{2^{-1}}(x_1)$, either $x_1$ or $y_1$ are in $B_2(z)\setminus B_{\frac{1}{4}}(z)$.

We repeat the process by covering $B_{2^{-i}}(x_i)\cap \cS(u)$ with balls of radius $2^{-i-1}$. Since singular points are finite, in a finite number $\bar i$ of steps we obtain $S_{\bar i} =1$, hence we stop. Evidently we have the estimate:
\begin{gather*}
 S_0 \leq c(m)^{\abs T}
\end{gather*}

In order to get a bound on $\abs T$, consider the singular point $x_{\bar i}$. If $T(i)=1$, then by construction there exists a singular point $z_i$ such that
\begin{gather*}
 2^{-i-1}\leq d(z_i,x_{\bar i})\leq 2^{-i+2}\, .
\end{gather*}
The bound on the number of bad scales ensures that $\abs{T}\leq 3C(p,\Lambda,N)$. This concludes the proof.
\end{proof}

\section{Stationary maps}
The study of the regularity of $p$-stationary harmonic maps is a little more complicated than in the minimizing case. There are two important differences: first of all, a sequence of $p$-stationary maps which is $W^{1,p}$ weakly convergent may not converge strongly (as opposed to the minimizing case). For this reason, we generalize the study of the defect measure carried out in \cite{lin_stat}.

It is worth mentioning also the work \cite{lin_p}, where the author studies the regularity of a class of $p$-minimizing functions. Some of the results available in this article are similar to the results we get here, for example the fact that the defect measure can be nonzero only if $p$ is an integer.

Moreover, in the stationary case a full-blown $\epsilon$-regularity theorem like \ref{th_eps_pHL} has not been proved yet, 
%
%
%
even though it seems very plausible to be valid. Note that, only for $p=2$, this problem has been completely solved by Bethuel in \cite{beth} (see also \cite{RS}), however the gauge techniques used in these papers are not easily adapted for generic $p$.

\subsection[e-regularity theorem]{$\epsilon$-regularity theorem}
Some partial results are available in literature under stricter assumptions. For example, see \cite{p-fuchs}, \cite{take_eps}, \cite{strze_eps}, \cite{mouya}, \cite{strze_eps_revenge}, \cite{Rtoc}. In \cite{take_eps} an $\epsilon$-regularity theorem is proved assuming that the target the standard sphere, and in \cite{Rtoc} under the strong assumption that the map is $W^{2,p}$. To the best of our knowledge, the most general result in this direction is the following, which assumes homogeneity of the target space.

\begin{theorem}\cite[Corollary 3.2]{torowang}\label{th_eps_stat}
 Let $N$ be a smooth compact homogeneous space with a left invariant metric. Then there exists $\epsilon(m,N,p),\alpha(m,N,p)>0$ such that if $u$ is a $p$-stationary harmonic map with $\theta(x,2r)<\epsilon$, then $u\in C^{1,\alpha} (B_x(r))$.
\end{theorem}
As an immediate corollary, we can obtain also estimates on the (scale-invariant) $C^{1,\alpha}$ norm of $u$.
\begin{corollary}\label{cor_eps_stat}
 Let $N$ be as above, and let $u:B_2(0)\subset \R^m\to N$ be a $p$-stationary harmonic map. There exist positive constants $\epsilon(m,N,p)$ and $\alpha(m,N,p)$ such that if $\theta(x,2r)<\epsilon$, then
 \begin{gather*}
  \norm{u}_{x,r} = r\norm{\nabla u}_{C^0(B_r(0))} + r^{1+\alpha} \sup_{x,y\in B_r(x)}\cur{\frac{\abs{\nabla u(x)-\nabla u(y)}}{\abs{x-y}^{\alpha}}   }  \leq 1\, .
 \end{gather*}
\end{corollary}
\begin{proof}
 The proof is obtained through a simple contradiction and compactness argument.
\end{proof}

Using a simple covering argument (see \cite[section 2.4.3]{EG}), we can obtain the following regularity theorem.
\begin{theorem}\cite[Theorem 2]{torowang}\label{th_statreg}
 Let $u$ be a stationary $p$-harmonic map $u:B_2(0)\to N$, where $N$ is a compact homogeneous manifold with a left invariant metric. Then for some $\alpha(n,p,N)>0$, $u\in C^{1,\alpha}(B_1(0)\setminus Z)$, where $Z$ is a closed set with $\cH^{m-p}(Z)=0$. In particular, if $p\geq n$, then $u$ is a $C^{1,\alpha}$ function on the whole domain.
\end{theorem}
\begin{remark}
Note that this result is not quantitative, meaning that there is no upper bound on $\abs{\nabla u}$ of any kind.

Indeed, even if $u\in C^{1,\alpha}(B\setminus Z)$, there is no uniform local bound on $\abs{\nabla u}$ on $B\setminus Z$. A counterexample can be found in \cite[Example 1.1]{lin_stat}. Let $u$ be a a nonconstant stationary $m$-harmonic map from $\R^m$, which has no singular points. Since such maps are conformal invariant in $\R^m$, it is easy to build a sequence $u_i$ with $m$-energy independent of $i$ such that $u_i\wto const$ in $W^{1,m}$ but
\begin{gather*}
 \abs{\nabla u}^m dV \wto c \delta_0\, ,
\end{gather*}
where the convergence is weak in the sense of measure. Evidently, in such a situation there can be no uniform upper bound on $\abs{\nabla u_i}$.
\end{remark}

However, one can easily tweak the previous argument to get effective $C^{1,\alpha}$ away from a set of Minkowski dimension $m-p$.
\begin{theorem}\label{th_p_mink}
 Let $u:B_2(0)\to N$ be a stationary $p$-harmonic map, where $N$ is a compact homogeneous manifold with a left invariant metric. Then
 \begin{gather*}
  \Vol\ton{\cB_r(u)} \leq C(m,N,p)r^{p} \int_{B_2(0)} \abs{\nabla u}^p \, .
 \end{gather*}
\end{theorem}
\begin{proof}
 The theorem is an easy consequence of the inclusion
\begin{gather*}
 \cB_r(u)\subset \cur{x\in B_1(0) \ \ s.t. \ \ \theta(x,r)\geq \epsilon}\, .
\end{gather*}
Let $B_r(x_i)$ be pairwise disjoint balls with centers in $\cB_r(u)$ such that $\cB_r(u)\subset \bigcup_{i} B_{5r}(x_i)$. Then the number $N$ of such balls is bounded above by
\begin{gather*}
 N \epsilon r^{m-p} \leq \sum_i \int_{B_{r}(x_i)} \abs{\nabla u}^p dV \leq \Lambda\, ,
\end{gather*}
and the thesis follows immediately.
\end{proof}
The example before shows that this result is in some sense sharp.

The aim of the following sections is to prove that the following result can be improved if $p$ is not an integer, or else if there exists no continuous nonconstant stationary $p$-harmonic map from $S^p$ into $N$.

Since stationary $m$-harmonic maps into symmetric targets are regular if the domain has dimension $m$, the following theorem about removable singularities should not be surprising.
\begin{theorem}\label{th_rem}\cite[Theorem 5.1]{mouya}
 Let $u:B_1(0)\setminus\{0\}\subset \R^m \to N$ be an $m$-harmonic map in $C^{1}(B_1(0)\setminus \{0\})$. If $u\in W^{1,m}(B_1(0))$, the singularity in $0$ is removable.
\end{theorem}

\subsection{The defect measure}
As we have seen, weak convergence of stationary maps does not imply strong convergence. The defect measure studied in \cite{lin_stat} gives a quantitative tool to measure how far the convergence is from being strong. In this section we study some of the properties of the defect measure. Most of the results are easy generalizations of the equivalent results available in \cite[section 1]{lin_stat} for the $p=2$ case, thus sometimes we will refer the reader to this article for the proofs.

The aim of this section is to show that the defect measure is absolutely continuous wrt the $\cH^{m-p}$ Hausdorff measure, and that it satisfies all the properties needed in order to apply the Federer's dimension reduction argument (see \cite[Appendix A]{sim_stat}).
\begin{remark}
 Throughout this section the $\epsilon$-regularity theorem \ref{th_eps_stat} will be an essential tool. Thus we will always assume to work with $p$-stationary maps $u:B_3(0)\to N$, where the target space $N$ is a compact homogeneous manifold with a left invariant metric.
\end{remark}

Let $\cH(\Lambda)$ be the set of stationary $p$-harmonic maps $u:B_2(0)\to N$ such that $\theta_u(x,2)\leq \Lambda$ for all $x\in B_1(0)$, and $\overline{\cH(\Lambda)}$ be its weak closure in the $W^{1,p}$ sense (recall that in this case the weak closure coincides with the weak sequential closure). Since $\theta(x,2)\leq (3/2)^{m-p} \theta(0,3)$, it is easy to see that
\begin{gather*}
 \theta_u(0,3)\leq \ton{\frac 2 3}^{m-p}\Lambda \quad \Longrightarrow \quad  u\in \cH(\Lambda)\, .
\end{gather*}

Consider a sequence $u_i\in \cH(\Lambda)$ and the corresponding sequence of measures $\abs{\nabla u_i}^p dV$. Given the uniform bound on the $p$-energies of $u_i$, up to passing to a subsequence, we can write that $u_i\wto u$ in the weak $W^{1,p}$ sense, and also that $\abs{\nabla u_i}^p dV \wto d\mu$ in the sense of weak convergence of measures. Note that by Fatou's lemma we can write
\begin{gather*}
 \abs{\nabla u_i}^p dV \wto d\mu = \abs{\nabla u}^p dV + d\nu\, ,
\end{gather*}
where $\nu$, a nonnegative Radon measure, is defined to be the defect measure.

Let $\cM(\Lambda)$ be the set of Radon measures $d\mu$ which can be obtained in this way. Note that $\cM(\Lambda)$ is closed under weak convergence in the sense of measures.

Following the study of the defect measure in \cite{lin_stat}, we generalize the results of this article to generic $p\in (1,\infty)$, and not only $p=2$. Since all the proofs in this section are similar to the ones in \cite{lin_stat}, we will sketch only the more complex ones.

\begin{theorem}\label{th_m-p}
 Let $u_i$ be a sequence in $\cH(\Lambda)$ such that $u_i\wto u$ in $W^{1,p}$ and $\abs{\nabla u_i}dV \wto d\mu$. Define the set
 \begin{gather*}
  \Sigma = \bigcap_{r>0} \cur{x\in \overline{B_1(0)} \ \ s.t. \ \ \liminf_{i\to \infty} r^{p-m}\int_{B_r(x)} \abs{\nabla u_i}^p dV > \epsilon }\, ,
 \end{gather*}
where $\epsilon=\epsilon(m,N,p)$ is chosen according to Theorem \ref{th_eps_stat}. Then
\begin{enumerate}
 \item $\Sigma$ is a closed subset of $B_1(0)$,
 \item $\Sigma$ has bounded $m-p$ Minkowski content, more precisely
 \begin{gather*}
  \Vol \ton{B_r(\Sigma)} \leq C(m,N,p,\Lambda)r^{m-p}\, ,
 \end{gather*}
 \item $\Sigma= \operatorname{supp}(\nu) \cup \operatorname{sing}(u)$, where $\operatorname{sing}(u)=\cur{x\in B_1(0) \ \ s.t. \ \ u \text{ is not } C^{1,\alpha} \text{ around } x}$ is the singular set of $u$,
 \item $d\nu$ is absolutely continuous wrt $\cH^{m-p}$. Moreover for almost all $x\in \Sigma$ wrt $\cH^{m-p}$, $d\nu = f(x) \cH^{m-p}|_{\Sigma}$ where $\epsilon \leq f(x)\leq C(n,\Lambda)$,
\end{enumerate}
\end{theorem}
\begin{proof}
 The proof of this theorem is based on standard covering arguments and the monotonicity of the normalized $p$-energy for stationary harmonic maps, which in turn easily yields the monotonicity of the quantity $\theta_\mu(x,r)= r^{p-m}d\mu(B_r(x))$. In the following, we sketch the main arguments in the proof. For more details, we refer the reader to \cite[Lemma 1.5 and Lemma 1.6]{lin_stat}.

 Point (1) follows easily from the $\epsilon$-regularity theorem. Indeed, if $x\not \in \Sigma$, then there exists a positive $r_x$ such that $\theta_\mu(x,2r_x)<\epsilon$. This implies that $u_i$ has uniform $C^{1,\alpha}$ bounds on $B_{r_x}(x)$, and thus $B_{r_x}(x)\cap \Sigma =\emptyset$.
 
 The uniform volume bound can be obtained by the same covering argument used in the proof of Theorem \ref{th_p_mink}.
 
 As for point (3), if $x\in B_1(0)\setminus \Sigma$, then the uniform $C^{1,\alpha}$ bounds given by the $\epsilon$-regularity theorem imply that $u_i$ converge in the $C^1$ sense to $u$. Thus $u$ is $C^{1,\alpha}$ around $x$ and $x\not \in \operatorname{supp}(\nu) \cup \operatorname{sing}(u)$. On the other hand, if $x\in \Sigma \setminus \operatorname{sing}(u)$, then there exists a radius $r_x$ small enough such that for all $s\leq r_x$, $s^{p-m}\int_{B_s(x)} \abs{\nabla u}^p <\epsilon/4$. Thus $s^{p-m}\nu(B_s(x))>0$, and so $x\in \operatorname{supp}(\nu)$.
 
 The last point follows from the monotonicity of the energy. Indeed, for all $x\in B_1(0)$ and $r<1$, we have $r^{p-m}\mu(B_r(x))\leq \mu(B_1(x)) \leq \Lambda$, thus $\mu$ is absolutely continuous wrt $\cH^{m-p}$. In particular, there exists a function $f$ such that $\mu=f(x)\cH^{m-p}|_\Sigma$.

 Moreover, by \cite[section 2.4.3]{EG}, $\limsup_{r\to 0} r^{p-m}\int_{B_r(x)} \abs{\nabla u}^p =0$ for $\cH^{m-p}$ a.e. $x\in \Sigma$. Thus we obtain the thesis.

\end{proof}

By Proposition \ref{prop_mono_stat}, it is easy to see that if $\theta_u(x,r)=\theta_u(x,0)$, then $u$ is a homogeneous function on $B_r(x)$. The next lemma, which is an immediate generalization of \cite[Lemma 1.7]{lin_stat}, shows that the same property holds for any measure $d\mu\in \cM(\Lambda)$.
\begin{lemma}\label{lemma_homdef}
 Let $u_i$ be a sequence in $\cH(\Lambda)$ such that $u_i\wto u$ in $W^{1,p}$ and $\abs {\nabla u_i}^p dx \wto d\mu = \abs{\nabla u}^p dx + d\nu$. Suppose also that for some $r_i\to 0$,
 \begin{gather*}
  \theta^{u_i} (0,1)-\theta^{u_i} (0,r_i) \to 0\, .
 \end{gather*}
Then both $\mu$ and $\nu$ are homogeneous measures, meaning that
\begin{gather*}
 d\mu = r^{m-p-1} dr d\sigma(\theta) \, ,
\end{gather*}
where the measure $\sigma $ is invariant wrt $r$, and $\partial_r u=0$ for a.e. $r\in (0,\infty)$.
\end{lemma}

\begin{proof}
Consider the measures $\abs{\nabla u_k}^p dV = \abs{\nabla u_k}^p r^{m-1} dr d\theta \equiv r^{m-p-1} dr d\sigma_k(r,\theta)$. By the monotonicity formula, the limit function $u$ is homogeneous because $\partial_r u=0$ a.e. away from the origin. Thus $\mu$ is homogeneous if and only if $\nu$ is homogeneous.

We want to prove that for almost every $r,R$, and every smooth test function $\phi: S^{m-1}\to \R$
\begin{gather*}
 \int_{S^{m-1}} \phi(\theta) d\sigma(r,\theta)=\int_{S^{m-1}} \phi(\theta) d\sigma(R,\theta)\, .
\end{gather*}

In order to do that, let $\psi$ a standard mollifier, i.e., let $\psi$ be a function such that $\operatorname{supp}(\psi)\subset [-1,1]$, $\psi\geq 0$ and $\int_\R \psi =1$.

For $a>\epsilon$, define the functions
\begin{gather*}
 \psi_ {a,\epsilon}(x) = \frac 1 \epsilon \psi\ton{\frac {\abs x - a} \epsilon }\, ,\\
 E_k(a,\epsilon)=\int_{0}^\infty\int_{S^{n-1}} \phi(\theta) \psi_{a,\epsilon} d\sigma_k(r,\theta)dr\, .
\end{gather*}
Note that, for a.e. $r\in (0,\infty)$,
\begin{gather*}
 \lim_{\epsilon \to 0} E_k(a,\epsilon) =: E_k(a)= \int_{S^{n-1}} \phi(\theta) d\sigma_k(a,\theta)\, .
\end{gather*}

In order to prove that $d\sigma(r,\theta)$ is invariant wrt $r$, we will show that its derivative in $r$ is zero, at least in a weak sense. Consider that
\begin{gather}\label{eq_deae}
 \frac{d}{da} E_k(a,\epsilon) =  \int_0^\infty \int_{S^{n-1}} \phi(\theta) \frac{d}{da} \psi_{a,\epsilon} d\sigma_k(r,\theta)dr = -\int_0^\infty \int_{S^{n-1}} \phi(\theta) \partial_r\psi_{a,\epsilon} d\sigma_k(r,\theta)dr\, .
\end{gather}

Set for simplicity $\varphi(r,\theta) = \psi_{a,\epsilon}(r) \phi(\theta)$, and consider the vector field (which is smooth for $\epsilon < a$)
\begin{gather*}
 \xi^j(x) = \varphi(x) \abs x ^{p-m} x^j\, .
\end{gather*}
By equation \eqref{eq_p-stat},
\begin{gather*}
 \int \abs{\nabla u_k}^{p-2} \ton{\abs{\nabla u_k}^2 \delta_{ij} - p \nabla_i u_k \nabla_j u_k} \partial_i \xi^j dV =0\, ,
\end{gather*}
which immediately yields
\begin{align*}
& \int \abs{\nabla u_k}^{p-2} \ton{\abs{\nabla u_k}^2 \delta_{ij} - p \nabla_i u_k \nabla_j u_k} \ton{\partial_i \varphi \abs x ^{p-m} x^j + (p-m)\varphi \abs x^{p-m-2} x^i x^j + \varphi \abs x ^{p-m} \delta_{ij}} dV =0\, ,\\
&\nonumber \iint \abs{\nabla u_k}^p \partial_r \varphi r^{p-m+1} r^{m-1} dr d\theta -p \iint \abs{\nabla u_k}^{p-2} \ps{\nabla u_k}{\nabla \varphi} r^{p-m+1} \partial_r u_k r^{m-1} drd\theta+\\
&  -p(p-m) \iint  \varphi r^{p-m} \abs{\nabla u_k}^{p-2}\abs{\partial_r u_k}^2r^{m-1} drd\theta+0\iint \varphi r^{p-m} \abs{\nabla u_k}^p r^{m-1} drd\theta=0\, .
\end{align*}
Equivalently
\begin{gather*}
\nonumber \iint  \abs{\nabla u_k}^p r^p \partial_r \psi_{a,\epsilon} \phi(\theta) dr d\theta = p(p-m) \iint  \psi_{a,\epsilon} \phi(\theta) r^{p-1} \abs{\nabla u_k}^{p-2}\abs{\partial_r u_k}^2 drd\theta +\\
 +p \iint  \abs{\nabla u_k}^{p-2} \abs{\partial_r u_k}^2 \partial_r \psi_{a,\epsilon} \phi(\theta) r^{p} dr d\theta
 +p\iint \abs{\nabla u_k}^{p-2} \ps{\nabla_{S^{m-1}} u_k}{\nabla_{S^{m-1}}\phi(\theta) }_{S^{m-1}} \partial_r u_k \psi_{a,\epsilon} r^{p-2} dr d\theta\, .
\end{gather*}
By equation \eqref{eq_deae}, the derivative of $E$ can be expressed as
\begin{gather*}
\nonumber \frac{d}{da} E_k(a,\epsilon) = p(m-p) \iint \psi_{a,\epsilon} \phi(\theta) r^{p-1} \abs{\nabla u_k}^{p-2}\abs{\partial_r u_k}^2 drd\theta +\\
 - p\iint\abs{\nabla u_k}^{p-2} \partial_\theta u_k \partial_r u_k \psi_{a,\epsilon} \partial_\theta \phi(\theta) r^{p-2} dr d\theta
 + p\frac{d}{da} \iint \abs{\nabla u_k}^{p-2} \abs{\partial_r u_k}^2 \psi_{a,\epsilon} \phi(\theta) r^{p} dr d\theta \, .
\end{gather*}
Integrating this equation on $[s,t]$ we get
\begin{gather*}
\nonumber E_k(s,\epsilon)-E_k(t,\epsilon) = p(m-p)\int_t^s da \iint \psi_{a,\epsilon} \phi(\theta) r^{p-1} \abs{\nabla u_k}^{p-2}\abs{\partial_r u_k}^2 drd\theta +\\
 - p \int_t^s da \iint\abs{\nabla u_k}^{p-2} \partial_\theta u_k \partial_r u_k \psi_{a,\epsilon} \partial_\theta \phi(\theta) r^{p-2} dr d\theta +
  p\qua{ \iint \abs{\nabla u_k}^{p-2} \abs{\partial_r u_k}^2 \psi_{a,\epsilon} \phi(\theta) r^{p} dr d\theta}_s^t\, .
\end{gather*}
By letting $\epsilon\to 0$, we obtain (at least a.e. in $s$ and $t$):
\begin{gather*}
 E_k(s)-E_k(t) = p(m-p)\int_t^s da \int  \phi(\theta) a^{p-1} \abs{\nabla u_k}^{p-2}\abs{\partial_r u_k}^2 d\theta +\nonumber\\
 - p \int_t^s da \int\abs{\nabla u_k}^{p-2} \partial_\theta u_k \partial_r u_k \partial_\theta \phi(\theta) a^{p-2} d\theta +
  p\qua{ \int \abs{\nabla u_k}^{p-2} \abs{\partial_r u_k}^2 \phi(\theta) r^{p} d\theta}_s^t\, .
\end{gather*}
Note that, by \eqref{eq_p-der},
\begin{eqnarray}
&&0\leq p(m-p)\int_t^s da \int  \phi(\theta) a^{p-1} \abs{\nabla u_k}^{p-2}\abs{\partial_r u_k}^2 d\theta \leq \norm{\phi}_{\infty}(m-p) [\theta_k(t)-\theta_k(s)]\, ,\\
&&\abs{\int_t^s da \int\abs{\nabla u_k}^{p-2} \partial_\theta u_k \partial_r u_k \partial_\theta \phi(\theta) a^{p-2} d\theta } \nonumber\\
&&\leq\ton{\int_{B_t(0)\setminus B_s(0)} dV r^{p-m} \abs{\nabla u_k}^{p-2} \abs{\partial_r u_k}^2 }^{1/2} \ton{\int_{B_t(0)\setminus B_s(0)} dV r^{p-m} \abs{\nabla u_k}^{p-2} r^{-2}\abs{\partial_\theta u_k}^2 \abs{\partial_\theta \phi}^2 }^{1/2}\nonumber\\
&&\leq s^{(p-m)/2}\norm{\nabla \phi}_{\infty} \Lambda^{1/2} \ton{\theta_k(t)-\theta_k(s)}^{1/2}\, , \\
&&\abs{ \qua{ \int \abs{\nabla u_k}^{p-2} \abs{\partial_r u_k}^2 \phi(\theta) r^{p} dr d\theta}_s^t}\leq \norm{\phi}_{\infty} \ton{\abs{\theta_k'(t)}+\abs{\theta_k'(s)}}\, .
\end{eqnarray}

Thus we obtain that, for a.e. $s,t>0$,
\begin{gather*}
 \lim_{k\to \infty} \int \phi(\theta) \ton{d\sigma_k(t,\theta)-d\sigma_k(s,\theta)} = 0.
\end{gather*}
Let $\tau_a$ be a translation in the radial coordinate by $a$. This implies that for every $a$:
\begin{gather*}
 \lim_{k\to \infty} \tau_a(d\sigma_k dr) -d\sigma_k dr = 0\, \quad \Longrightarrow\quad \, \tau_a(d\sigma dr) =d\sigma dr \, .
\end{gather*}
Thus we have proved the invariance of the measure $d\sigma$, and in turn the homogeneity of $d\mu$ and $d\nu$.
\end{proof}

This lemma will play a crucial role in proving a generalization of Theorem \ref{th_eps_+} for stationary functions and in the dimension reduction properties of the space $\cM(\Lambda)$ explained in the following section.

\subsection{Dimension reduction arguments}
In this section, we show that the dimension reduction argument proved in \cite[Theorem A.4]{sim_stat} can be applied to the measures in $\cM(\Lambda)$. As a corollary, we will prove that if $p$ is not an integer, then there cannot be any defect measure, and if $p$ is an integer, $\cM(\Lambda)$ contains a constant multiple of $\cH^{m-p}|_L$, where $L$ is some $m-p$ dimensional subspace of $\R^m$.

\begin{definition}
 Given $\mu\in \cM(\Lambda)$, $y\in B_1(0)$ and $r\leq 2$, we define the Radon measure
 \begin{gather*}
  \mu_{y,r}(A)= r^{m-p} \mu(y+rA)\, .
 \end{gather*}
\end{definition}
It is clear from the definition that $\mu_{y,r}\in \cM(\Lambda)$ for every $r>0$ sufficiently small, and since $\cM(\Lambda)$ is closed under weak convergence of measure, given any sequence $r_k\to 0$, there always exists a subsequence such that $\mu_{y,r_{k_i}}\wto \mu_{y,0}\in \cM(\Lambda)$ (note that $\mu_{y,0}$ may depend on the sequence $r_{k_i}$).

\begin{definition}
 Let $\cF$ be the set of closed subsets of $\overline{B_1(0)}\subset \R^m$. Define the map $\pi:\cM(\Lambda)\to \cF$ by $\pi(\mu)=\Sigma$, where $\Sigma$ is the set defined in Theorem \ref{th_m-p}.
\end{definition}

The following lemma generalizes \cite[Lemma 1.7]{lin_stat} and is the key to proving the dimension reduction properties.
\begin{lemma}\label{lemma_17}
 Let $\mu\in \cM(\Lambda)$, $y\in B_1(0)$ and $\lambda \leq 2$. Then
 \begin{enumerate}
  \item $\cM(\Lambda)$ is closed under rescaling, meaning that $\mu_{y,\lambda}$ belongs to $\cM(\Lambda)$,
  \item given any sequence $\lambda_k \to 0$, there exists a subsequence $\lambda_{k_i}$ such that
  \begin{gather*}
   \mu_{y,\lambda_{k_i}}\wto \bar \mu\in \cM(\Lambda) \quad \text{ with } \quad \bar \mu_{0,r}=\bar \mu \ \ \ \forall r>0\, ,
  \end{gather*}
  \item $\pi(\mu_{y,\lambda})=\lambda^{-1}\ton{\pi(\mu)-y}$,
  \item if $\mu$ is absolutely continuous wrt the $n$-dimensional Lebesgue measure, then $\pi(\mu)=\emptyset$
  \item if $\mu_k\wto \mu$, then for every $\epsilon>0$, there exists $\bar k(\epsilon)$ such that for $k\geq \bar k$:
  \begin{gather*}
   \pi(\mu_k) \subset \cur{x\in \overline B_1(0) \ \ s.t. \ \ d(x,\pi(\mu))<\epsilon}\, .
  \end{gather*}
 \end{enumerate}
\end{lemma}
\begin{remark}
 Note that properties $1$ to $5$ coincide with properties A.1, A.2 and A.3 in \cite{sim_stat}.
\end{remark}
\begin{proof}
Properties $1$ and $3$ follow directly from the definitions given above, while property 4 is an easy consequence of the definition of $\pi(\mu)=\Sigma$ given in Theorem \ref{th_m-p}.

Property $2$ is a direct consequence of Lemma \ref{lemma_homdef} and the monotonicity of $\theta$. First of all, observe that
\begin{gather*}
 \theta_\mu(x,r)=r^{p-m} \mu(B_r(x))
\end{gather*}
is a monotone nondecreasing quantity for all $\mu\in \cM(\Lambda)$. Moreover $\theta_{\mu}(x,r)=\theta_{\mu_{x,r}} (0,1)$, and thus $\theta_{\bar \mu}(0,r)=\theta_{\bar \mu}(0,0)$ for all $r>0$.

Consider a sequence of functions $w_i\in \cH(\Lambda)$ such that $\abs{\nabla w_i}^p dV \wto \bar \mu$. The weak convergence implies that for all $\epsilon$ and $r>0$
\begin{gather*}
 \lim_{i\to \infty} \theta_{w_i}(0,1)=\lim_{i\to \infty} \int_{B_1(0)} \abs{\nabla w_i}^p dV \leq \bar \mu(B_{1+\epsilon}(0)) 
 =(1+\epsilon)^{m-p} \theta_{\bar \mu}(0,0)\, ,\\
 \lim_{i\to \infty} \theta_{w_i}(0,r)=\lim_{i\to \infty} r^{p-m}\int_{B_r(0)} \abs{\nabla w_i}^p dV \geq r^{p-m}\bar \mu(B_{r(1-\epsilon)}(0)) 
 =(1-\epsilon)^{m-p} \theta_{\bar \mu}(0,0)\, .
\end{gather*}
In other words, for every $r>0$ $\lim_{i\to \infty}\theta_{w_i}(0,1)-\theta_{w_i}(0,r)=0$, and property $2$ follows directly from Lemma \ref{lemma_homdef}.

As for property $5$, the proof is a simple application of the $\epsilon$-regularity theorem. Let $\mu_i$ be a sequence of measures in $\cM(\Lambda)$, and consider the sequence of compact sets $\pi(\mu_i)$. By Hausdorff compactness principle, up to passing to a subsequence, $\pi(\mu_i)\to E$, where $E$ is a closed set and the convergence is the Hausdorff convergence in $\R^m$. This in particular implies that for every $\epsilon>0$, there exists $\bar k(\epsilon)$ such that for $k\geq \bar k$:
  \begin{gather*}
   \pi(\mu_k) \subset \cur{x\in \overline B_1(0) \ \ s.t. \ \ d(x,E)<\epsilon}\, .
  \end{gather*}
We are left to prove that $E\subset \pi(\mu)$. Let $x\in E$, then there exists a sequence $x_i\in \pi(\mu_i)$ such that $x_i\to x$ in the usual Euclidean sense. By definition of $\pi(\mu)$, $\theta_{\mu_i}(x,0)>\epsilon$, and by monotonicity of $\theta$, for all $r>0$ and for all $i$, $\theta_{\mu_i}(x_i,r)>\epsilon$.

This immediately implies that for all $\delta>0$ and for all $r>0$:
\begin{gather*}
 \theta_{\bar \mu}(x,r+\delta)=(r+\delta)^{p-m} \bar \mu(B_{r+\delta}(x)) \geq \ton{\frac{r+\delta}{r}}^{p-m}\lim_{i\to \infty} \theta_{\mu_i}(x_i,r) > \ton{1+\frac{\delta} {r} }^{p-m} \epsilon\, .
\end{gather*}
Thus we can conclude that $\theta_{\bar \mu}(x,0)>0 \ \Longleftrightarrow\ \theta_{\bar \mu}(x,0)>\epsilon$, and thus $x\in \pi(\bar \mu)$.

\end{proof}

As an application of this lemma, we can apply the dimension reduction argument in \cite[Appendix A]{sim_stat} and prove that if there exists a nonzero defect measure, then $\cM(\Lambda)$ contains a measure $\mu$ which is exactly a constant multiple of the $m-p$ Hausdorff measure on an $m-p$ dimensional subspace of $\R^m$. As a corollary, we obtain that there cannot be any nonzero defect measure if $p$ is not an integer.
\begin{proposition}\label{prop_nicelim}
 Suppose that there exists some sequence $u_i\in \cH(\Lambda)$ such that $\abs{\nabla u_i}^p dV \wto \abs{\nabla u}^p dV +d\nu$, where $d\nu\neq 0$. Then $p$ must be an integer, and there exists a sequence $w_i\in \cH(\Lambda)$ such that
 \begin{gather*}
  w_i\wto const\,  , \quad \quad \abs{\nabla w_i}^pdV \wto d\nu\, ,
 \end{gather*}
where $d\nu$ is a constant multiple of the $m-p$ Hausdorff measure on a $m-p$ subspace of $\R^m$.
\end{proposition}
\begin{proof}
By point (4) in Theorem \ref{th_m-p}, the measure $d\nu$ is absolutely continuous wrt $\cH^{m-p}$ and nonzero. Thus there exists a point $x\in \Sigma$ with positive $m-p$ density (see \cite[2.10.19]{Fed}). Specifically we have
\begin{gather}\label{eq_hinfty}
 \limsup_{r\to 0} \frac{\cH^{m-p}_\infty (\Sigma \cap B_r(x))}{r^{m-p}} >0 \quad \Longrightarrow \quad \exists \lambda_k\to 0 \ \ s.t. \ \ \lim_{k\to \infty} \frac{\cH^{m-p}_\infty (\Sigma \cap B_{\lambda_k}(x))}{\lambda_k^{m-p}} >0\, ,
\end{gather}
where $\cH^{m-p}_r (A)= \inf\cur{\sum_{j=1}^\infty \omega_{m-p}\ton{\frac{\diam(C_j)}{2}}^{m-p} \ \ s.t. \ \ C_j\subset \R^m \ \ \text{and}\ \ A\subset \cup_j C_j \ \ \text{and}\ \ \diam(C_j)\leq r}$.

By Lemma \ref{lemma_17}, up to passing to a subsequence, $\mu_{x,\lambda_k}\wto \bar \mu$, where $\bar \mu$ is homogeneous (and thus we can extend the definition of $\bar \mu$ to the whole $\R^m$).

Let $\bar \Sigma=\pi(\bar \mu)$ be the singular set of $\bar \mu$. We are going to show that this set must have positive $m-p$ Hausdorff measure. Indeed, suppose by contradiction that $\cH^{m-p}(\bar \Sigma)=0$, which is equivalent to $\cH^{m-p}_\infty (\bar \Sigma)=0$.
Then for every $\delta>0$ there exists a family of balls $B_{\rho_i}(z_i)=B_i$ such that $\bar \Sigma \subset \cup_i B_i$ and $\sum_i \rho_i^{m-p}\leq \delta$. Note that $\bar \Sigma$ is a compact set, thus, by Lemma \ref{lemma_17}, for all $k$ sufficiently large such that also $\Sigma_k=\pi(\mu_{x,\lambda_k})$ is contained in $\cup_i B_i$. Since $\pi(\mu_{x,\lambda_k})=\lambda_k^{-1} (\pi(\mu)-x)=\lambda_k^{-1} (\bar \Sigma-x)$, this contradicts \eqref{eq_hinfty}.

Define the set $S$ to be the invariant subspace of $\bar \mu$, i.e.,
\begin{gather*}
 S=\cur{y\in \R^m \ \ s.t. \ \ \bar \mu_{y,\lambda}=\bar \mu \ \ \forall \lambda>0}\, .
\end{gather*}
It is evident that $0\in S$. Moreover, by homogeneity of $\bar \mu$, $S$ is a vector subspace of $\R^m$.

Let $d\in \N$ be its dimension. If $d<m-p$, then there exists a point $x\in \bar \Sigma\setminus S$ with positive $m-p$ density. Let $r_k\to 0$ be such that $\bar \mu_{x,r_k}$ converges weakly to some measure $\mu'$ with $\cH^{m-p}(\pi(\mu'))>0$.

For all $y\in S$, $\bar \mu_{x+y,\lambda}=\bar \mu_{x,\lambda}$, and so $\mu'_{1,\lambda}=\mu'$. This proves that $S$ is an invariant subspace for $\mu'$ as well. Moreover, also $x$ belongs to the invariant space of $\mu'$. Indeed
\begin{gather*}
 \mu'_{x,1} = \lim \bar \mu_{x+r_k x,r_k} = \lim \bar \mu_{x,r_k/(1+r_k)} = \mu'\, ,
\end{gather*}
where the limits are in the weak measure sense. Note that $\theta_{\mu'}(0,r)=\theta_{\bar \mu}(x,0)>\epsilon$ for all $r$, thus $0$ is a singular point for $\mu'$.

Thus, if $d<m-p$, then there exists $\mu'\in \cM$ such that its invariant subspace $S'$ has dimension $d+1$ and all points in $S$ are singular points.

\paragraph{If $p$ is not an integer}
By applying induction on $d$ to the previous argument, we can find a measure $\mu\in \cM$ with an invariant set $S$ of dimension $m-\floor p > m-p$ containing only singular points. This contradicts the fact that the singular set of $\mu'$ must have Hausdorff dimension $m-p$. Thus, as long as $p$ is not an integer, there cannot be any nonzero defect measure. Moreover, the singular set of all $\mu\in \cM$ must have zero $m-p$ Hausdorff measure, and actually its Hausdorff dimension must be $\leq m-\ceil p$.

\paragraph{If $p$ is an integer}
By applying induction on $d$ to the previous argument, we can find a measure $\mu\in \cM$ with an invariant set $S$ of dimension $d=m-p$ containing only singular points. Note that the singular set $\Sigma$ of $\mu$ coincides with $S$. Indeed, $S\subset \Sigma$, and if there existed some $x\in \Sigma\setminus S$, then we could apply the blow-up arguments discussed above to obtain a homogeneous measure $\mu'\in \cM$ with invariant subspace $S'$ of dimension $d+1$ with $S'\subset \Sigma'$, which is impossible.

Now consider a sequence $u_i\in \cH(\Lambda)$ with $\abs{\nabla u_i}^p dV \wto d\mu$ and $u_i\wto u$ in the weak $W^{1,p}$ sense. It is easy to see that for every $\epsilon,r>0$ and every $x\in S$:
\begin{gather*}
 \limsup_k \theta_{u_k}(x,1)\leq \int_{B_{1+\epsilon}(x)} d\mu = (1+\epsilon)^{m-p} \theta_\mu (x,r)=(1+\epsilon)^{m-p} \theta_\mu (0,0)\, ,\\
 \liminf_k \theta_{u_k}(x,r)\geq r^{m-p} \int_{B_{r(1-\epsilon)}(x} d\psi = (1-\epsilon)^{m-p} \theta_\mu (x,0)=(1-\epsilon)^{m-p} \theta_\mu (0,0)\, .
\end{gather*}
Thus for each $x\in S$, there exists a sequence $r_k\to 0$ such that $\theta_{u_k}(x,1)-\theta_{u_k}(x,r_k)\to 0$. By Lemma \ref{lemma_homdef}, both $u$ and the defect measure $\nu$ are homogeneous wrt every point $x\in S$, and thus $S$ is an invariant set for both $u$ and $d\nu$.

In particular, $u$ induces a homogeneous $p$-harmonic map $u:\R^p\setminus \{0\}\to N$ with finite $p$-energy. By the removable singularity Theorem \ref{th_rem}, $u$ can be extended to a $C^{1,\alpha}$ map on the whole $\R^p$. Moreover, since this map is continuous and homogeneous, it has to be constant.

As for the measure $d\nu$, its support must be the invariant subspace $S$, and thus $d\nu(A)= c \cH^{m-p}(A\cap S)$, where $c$ is either $0$ or some constant $>\epsilon$.

\end{proof}

\subsection{\texorpdfstring{Defect measure and $p$-harmonic spheres for integer $p$}{Defect measure and p-harmonic spheres for integer p}}

Here we study the case where $p$ is an integer, following the analysis made by Lin in \cite{lin_stat}.

We want to show that
\begin{proposition}
 If there exists a nonzero defect measure, then there exists a nonconstant $C^{1,\alpha}$ $p$-harmonic map $v:S^p\to N$. As a corollary, if such a map does not exist then regularity of stationary $p$-harmonic maps improves.
\end{proposition}

\begin{remark}
 As the referee pointed out to us, this proposition has already been proved in \cite{wang_changyou}, where the author studies limits of solutions to the generalized Ginzburg-Landau functional. Also in this article, the technique is based on \cite{lin_stat}. For the sake of completeness, here we present a similar proof.
\end{remark}

\begin{proof}
Let $w_i$ be one of the sequences of maps in $\cH(\Lambda)$ given by Proposition \ref{prop_nicelim}, such that $w_i\wto const$ and $\abs{\nabla w_i}^p dV \wto d\nu$, 
where $d\nu$ is a constant multiple of the $m-p$ Hausdorff measure on a $m-p$ subspace of $\R^m$ (say $\R^{m-p}\subset\R^{m-p}\times\R^p$). 
Let $x_0=0$ and $x^i$, $i=1,\cdots,m$  be the canonical basis for $\R^m$. 
Since the defect measure is a constant multiple of $\cH^{m-p}|_{\R^{m-p}}$, for all $0<r<R$ and $k=0,\dots,(m-p)$ it holds 
\begin{gather*}
\theta_\nu(x^k,r)=\theta_\nu(x^k,R). 
\end{gather*}
Accordingly, the monotonicity formula \eqref{prop_mono_stat} gives
\begin{gather}\label{eq_0}
 \lim_{i\to\infty} \theta_{w_i}(x^k,R)-\theta_{w_i}(x^k,r) =  \lim_{i\to\infty} p\int_{B_R(x^k)\setminus B_r(x^k)}\abs{y-x^k}^{p-m}\abs{\nabla w_i}^{p-2}\abs{\frac{\partial w_i}{\partial n_k}}^2 dV(y)=0,
\end{gather}
where $\partial_{n_k}$ is the exterior normal derivative with respect to the point $x^k$.

For any $k=1,\cdots,(m-p)$, it is easy to see that for all $f$:
\begin{gather*}
\frac{\partial f }{\partial x^k}(y)= \abs{y-x_0}\frac{\partial f }{\partial n_0 } u(y) - \abs{y-x^k} \frac{\partial f }{\partial n_k} u(y)\, .
\end{gather*}
Fix any $r>0$, then
\begin{gather*}
 \int_{B_1(0)}\abs{\nabla w_i}^{p-2}\abs{\frac{\partial w_i}{\partial x^k}}^2 dV = \int_{A_r}\abs{\nabla w_i}^{p-2}\abs{\frac{\partial w_i}{\partial x^k}}^2 dV +\int_{B_r(0)}\abs{\nabla w_i}^{p-2}\abs{\frac{\partial w_i}{\partial x^k}}^2 dV +\int_{B_r(x^k)}\abs{\nabla w_i}^{p-2}\abs{\frac{\partial w_i}{\partial x^k}}^2 \, ,
\end{gather*}
where $A_r = B_1(0)\setminus \left(B_r(0) \cup B_r(x^k)\right)$. As $i$ goes to infinity, the first integral converges to zero because by \eqref{eq_0}
\begin{gather*}
 \frac 1 2 \int_{A_r}\abs{\nabla w_i}^{p-2}\abs{\frac{\partial w_i}{\partial x^k}}^2 dV\leq \nonumber\\
 \leq \int_{B_2(0)\setminus B_r(0)}\abs{\nabla w_i}^{p-2}\abs{y-0}^2 \abs{\partial_{n_0} w_i}^2 dV + \int_{B_2(x^k)\setminus B_r(x^k)}\abs{\nabla w_i}^{p-2}\abs{y-x^k}^2 \abs{\partial_{n_k} w_i}^2 dV \to 0\, .
\end{gather*}
As for the second integral, we can estimate
\begin{gather*}
\nonumber \int_{B_r(0)}\abs{\nabla w_i}^{p-2}\abs{\frac{\partial w_i}{\partial x^k}}^2 dV \leq \int_{B_r(0)}\abs{\nabla w_i}^{p-2}\abs{y-0}^2 \abs{\partial_{n_0} w_i}^2 dV + \int_{B_r(0)}\abs{\nabla w_i}^{p-2}\abs{y-x^k}^2 \abs{\partial_{n_k} w_i}^2 dV \leq \\
 \leq r^2 \int_{B_r(0)} \abs{\nabla w_i}^2 + 4 \int_{B_2(x^k)\setminus B_r(x^k)}\abs{\nabla w_i}^{p-2}\abs{\partial_{n_k} w_i}^2 dV\, .
\end{gather*}
In a similar way, we can estimate the third integral. Since $r>0$ is arbitrary, we obtain that for every $k=1,\cdots,(m-p)$
\begin{gather}\label{3.1}
 \lim_{i\to \infty} \int_{B_1(0)}\abs{\nabla w_i}^{p-2}\abs{\frac{\partial w_i}{\partial x^k}}^2 dV =0\, .
\end{gather}

We now proceed as in in \cite[Lemma 3.1]{lin_stat}. Set $X_1=(x_1,\dots,x_{m-p})$, $X_2=(x_{m-p+1},\dots,x_m)$, and
\begin{gather*}
f_i(X_1)= \sum_{k=1}^{m-p}\int_{B^p(0,1/2)}\abs{\nabla w_i}^{p-2}\abs{\frac{\partial w_i}{\partial x^k}}^2(X_1,X_2)dX_2,
\end{gather*}
defined on $B^{m-p}(0,1/2)$. By \eqref{3.1} $f_i\to 0$ in $L^1(B^{m-p}(0,1/2))$. 
Theorem \ref{th_statreg} ensures that $w_i$ is $C^{1,\alpha}$ in a neighborhood of $\{X_1\}\times B^p(0,1/2)$ for $\cH^{m-p}$-a.e. point $X_1\in B^{m-p}(0,1/2)$. In particular we can choose a sequence $\{X_1^i\}_{i=1}^\infty$ of such points. 
The weak-$L^1$ estimate for the Hardy-Littlewood maximal function says that
\begin{gather*}
\abs{\left\{\sup_{r>0}\frac{1}{\abs{B^{m-p}(X_1^i,r)}}\int_{B^{m-p}(X_1^i,r)}f_i(X_1)dX_1>\lambda\right\}}< \frac{C(m-p)}{\lambda}\left\|f_i\right\|_{L^1(B^{m-p}(0,1/2))}
\end{gather*}
for all positive $\lambda$. Then
\begin{gather}\label{3.3}
\sup_{r>0}r^{p-m}\int_{B^{m-p}(X_1^i,r)}f_i(X_1)dX_1\to 0,\ \textrm{as }i\to\infty.
\end{gather}
Let $\epsilon_0>0$ be such that Corollary \ref{cor_eps_stat} works on $B_3(0)$ with $r=3/2$ and let $c(n)$ be a dimensional constant chosen in such a way that $B^{m-p}(0,3)\times B^p(0,3)$ can be covered with $c(n)/2$ balls of radius $1/2$. Fix $\delta>0$. 
Since there exists a nonzero defect measure, then $\abs{\nabla w_i}$ can not be uniformly bounded on $B^{m-p}(X_1^i,\delta/2)\times B^{p}(0,\delta)$. 
Hence by Corollary \ref{cor_eps_stat}
\begin{gather*}
\max_{X_2\in B^p(0,1/2)}\delta^{p-m}\int_{B^{m-p}(X_1^i,\delta)\times B^{p}(X_2,\delta)}\abs{\nabla w_i}^p dV \geq\epsilon_0,
\end{gather*}
for all $i$ large enough. On the other hand, since $w_i$ is $C^{1,\alpha}$ in a neighborhood of $\{X_1^i\}\times B^p(0,1/2)$, the $\epsilon$-regularity gives that
for every $i$ there exists $\delta(i)>0$ such that
\begin{gather*}
\delta^{p-m}\int_{B^{m-p}(X_1^i,\delta)\times B^{p}(X_2,\delta)}\abs{\nabla w_i}^p dV \leq \frac{\epsilon_0}{2c(n)},\ \forall 0<\delta<\delta(i),\ \forall X_2\in B^p(0,1/2).
\end{gather*}
Then for $i$ large enough we can find a sequence $\{\delta_i\}$ of positive numbers, $\delta_i\to 0$ as $i\to\infty$, such that 
\begin{gather}\label{max}
\max_{X_2\in B^p(0,1/2)}\delta_i^{p-m}\int_{B^{m-p}(X_1^i,\delta_i)\times B^{p}(X_2,\delta_i)}\abs{\nabla w_i}^p dV =\frac{\epsilon_0}{c(n)}.
\end{gather}
Moreover the maximum is achieved at some $X_2^i\in B^p(0,1/4)$, since otherwise for all $i$ large enough (such that $\delta_i<1/8$),
\begin{gather*}
\int_{B^{m-p}(0,1)\times (B^{p}(0,1/2)\setminus B^p(0,1/8))}\abs{\nabla w_i}^p dV\geq C(n,p,\epsilon_0)>0,
\end{gather*}
contradicting the assumption that $w_i\to const$ in $C^{1,\alpha}(B^{m-p}(0,1)\times (B^{p}(0,1/2)\setminus B^p(0,1/8)))$.\\
Now, set $Q_i=(X_1^i, X_2^i)$, $R_i=1/(4\delta_i)$ (so that $R_i\to\infty$ as $i\to\infty$) and define the $p$-stationary maps $v_i(y)=w_i(Q_i+\delta_iy)$ on $B^{m-p}(0,R_i)\times B^{m-p}(0,R_i)$. The convergence in \eqref{3.3} can be read as 
\begin{gather}\label{3.5}
V_i:=\sup_{0<R<2R_i}R^{p-m}\int_{B^{m-p}(0,R)\times B^p(0,2R_i)}\sum_{k=1}^{m-p}\abs{\nabla v_i}^{p-2}\abs{\frac{\partial v_i}{\partial x^k}}^2 dV\to 0,\ \textrm{as }i\to\infty.
\end{gather}
From \eqref{max} we deduce that
\begin{gather}\label{3.6}
\int_{B^{m-p}(0,1)\times B^{p}(0,1)}\abs{\nabla v_i}^p dV = \frac{\epsilon_0}{c(n)}\\
=\max_{Y_2\in B^p(-4R_iX_2^i,2R_i)} \int_{B^{m-p}(0,1)\times B^{p}(Y_2,1)}\abs{\nabla v_i}^p dV
=\max_{Y_2\in B^p(0,R_i-1)} \int_{B^{m-p}(0,1)\times B^{p}(Y_2,1)}\abs{\nabla v_i}^p dV
\end{gather}
Finally, since $w_i\in \cH(\Lambda)$ for all $i$, then for every $0<R<R_i$,
\begin{gather}\label{3.7}
\sup_i \int_{B^{m-p}(0,R)\times B^p(0,R)}\abs{\nabla v_i}^p dV \leq \Lambda R^{m-p}.
\end{gather}
Since $R_i$ is increasing, this latter ensures that, for every positive $R$, up to extract a subsequence $v_i$ weakly converges in $W^{1,p}$ on $B^{m-p}(0,R)\times B^p(0,R)$. 
Hence by a diagonalisation process we can find a map $v\in W^{1,p}_{loc}(\R^m,N)$ such that, up to extract a subsequence, $v_i\wto v$ in $W^{1,p}(B^{m-p}(0,R)\times B^p(0,R))$ for all $R>0$. 
Moreover, thanks to the lower semicontinuity of the $p$-energy
\begin{gather}\label{eq_fin_en}
\int_{B^{m-p}(0,R)\times B^p(0,R)}\abs{\nabla v}^p dV \leq \Lambda R^{m-p} \ \forall\, R.
\end{gather}
Let $\phi\in C^\infty_c(B^{m-p}(0,1)\times B^p(0,1))$ such that $0\leq\phi\leq 1$, $\phi\equiv 1$ in $B^{m-p}(0,3/4)\times B^p(0,1/2)$ and $\abs{\nabla\phi}<8$. Set
\begin{gather*}
F_i(a)=\int_{B^{m-p}(0,1)\times B^p(0,1)}\abs{\nabla v_i}^p(x+a)\phi(x) dV(x),
\end{gather*}
for $a\in B^{m-p}(0,3)\times B^p(0,R_i-1)$. The divergence formula \eqref{eq_p-stat}, H\"older inequality, \eqref{3.6} and \eqref{3.5}  give that
\begin{gather*}
\abs{\frac{\partial F_i}{\partial a_k}} = \abs{\int_{B^{m-p}(0,1)\times B^p(0,1)}\frac{\partial}{\partial x^k}\abs{\nabla v_i}^p(x+a)\phi(x) dV(x) }
\\=p\abs{\int_{B^{m-p}(0,1)\times B^p(0,1)}\abs{\nabla v_i}^{p-2}(x+a)\nabla_l v_i(x+a)\nabla_k v_i(x+a)\nabla^l \phi(x) dV(x) }
\\\leq 8p \int_{B^{m-p}(0,1)\times B^p(0,1)}\abs{\nabla v_i}^p(x+a) dV \int_{B^{m-p}(0,1)\times B^p(0,1)}\sum_{k=1}^{m-p}\abs{\nabla v_i}^{p-2}(x+a)\abs{\frac{\partial v_i}{\partial x^k}}^2(x+a) dV \to 0,\ \textrm{as }i\to\infty,
\end{gather*}
uniformly on compact sets, for each $k=1,\dots,(m-p)$. Then, for $i$ large enough 
\begin{gather*}
\int_{B^{n}(a,1/2)}\abs{\nabla v_i}^p(x) dV(x)\leq 2F_i(0)\leq\frac{2\epsilon_0}{c(n)},\ \forall\,a\in B^{m-p}(0,3)\times B^p(0,3),
\end{gather*}
and by the choice of $c(n)$ 
\begin{gather*}
\int_{B^{m-p}(0,3)\times B^p(0,3),}\abs{\nabla v_i}^p(Y_1,Y_2+b) dV(Y,_1,Y_2)\leq\epsilon_0,\ \forall\,b\in B^{p}(0,R_i-3).
\end{gather*}
Hence Corollary \ref{cor_eps_stat} yields that for all positive $R$, as $i\to\infty$, $v_i\to v$ up to a subsequence in $C^{1,\alpha'}(B^{n_p}(0,3/2)\times B^p(0,R))$. 
The limit map $v$ is a $C^{1,\alpha'}$ $p$-harmonic map defined on $B^{m-p}(0,3/2)\times\R^p$ which is non-constant since by strong convergence
\begin{gather*}
\int_{B^{m-p}(0,1)\times B^p(0,1)}\abs{\nabla v}^p(x) dV(x) = \frac{\epsilon_0}{c(n)}.
\end{gather*}
Moreover taking limits in \eqref{3.5} and \eqref{3.7} it is clear that
\begin{gather*}
\int_{B^{m-p}(0,R)\times\R^p}\sum_{k=1}^{m-p}\abs{\nabla v}^{p-2}\abs{\frac{\partial v}{\partial x^k}}^2 dV =0,
\end{gather*}
i.e., $v$ induces a nonconstant $C^{1,\alpha'}$ $p$-harmonic maps from $\R^p$ to $N$ which, thanks to \eqref{eq_fin_en}, has finite $p$-energy. 
By a conformal change, $v$ can be seen as a nonconstant, $C^{1,\alpha}$ $p$-harmonic map from $S^p\setminus {0}$ to $N$ with finite $p$-energy. Given the removable singularity theorem \ref{th_rem}, $v$ is a $C^{1,\alpha}$ $p$-harmonic map from the entire $S^p$ into $N$.
\end{proof}

\subsection{Regularity estimates}
As we have seen, an important difference between stationary and minimizing maps is that a weakly convergent sequence of stationary maps need not converge strongly, while this is true in the minimizing case. However, by analyzing the defect measure, we have concluded that
\begin{lemma}
 Let $u_i$ be a $W^{1,p}$ weakly convergent sequence of stationary $p$-harmonic maps $u_i:B_2(0)\to N$, where $N$ is a compact homogeneous space with a left invariant metric. If $p$ is not an integer, or if there are no nonconstant $C^1$ stationary $p$-harmonic maps from $S^p\to N$, then $u_i$ converges strongly to its limit, which is a stationary $p$-harmonic map.
\end{lemma}
This lemma allows us to reproduce all the results studied in the minimizing case, in particular Propositions \ref{prop_conespl}, \ref{th_qrig}, \ref{th_eps_+}, \ref{cor_eps_+}. Thus, under these assumptions, stationary $p$-harmonic maps enjoy the same regularity properties of minimizing maps.
\begin{theorem}
 Let $u:B_2(0)\to N$ be a stationary $p$-harmonic map, where $N$ is a smooth compact homogeneous space with a left invariant metric. If $p$ is not an integer, then for all $\epsilon>0$:
\begin{gather*}
  \Vol\ton{B_r(\cS(u))\cap B_1(0)}\leq \Vol\ton{B_r(\cB_r(u))\cap B_1(0)}\leq C(m,N,p,\epsilon) r^{\floor p +1-\epsilon}\, .
\end{gather*} 
Moreover, for any $p$ under the additional assumption \eqref{B}, we can improve the previous estimate to
\begin{gather*}
  \Vol\ton{B_r(\cS(u))\cap B_1(0)}\leq \Vol\ton{B_r(\cB_r(u))\cap B_1(0)}\leq C r^{b+2-\eta}\, .
 \end{gather*}  
\end{theorem}

As in the minimizing case, we get the following sharp integrability results.
\begin{corollary}
Under the hypothesis of the previous theorem, if $p$ is not an integer then for all $\epsilon>0$, $\nabla u\in L^{\floor{p}+1-\epsilon}(B_1(0))$ with
 \begin{gather*}
  \int_{B_1(0)} \abs{\nabla u}^{\floor{p}+1-\epsilon}\leq C(m,\Lambda,N,p,\epsilon)\, .
 \end{gather*}
Moreover, for all $p$ and under the additional assumption \eqref{B}, $\nabla u\in L^{b+2-\epsilon}(B_1(0))$ with
 \begin{gather*}
  \int_{B_1(0)} \abs{\nabla u}^{b+2-\epsilon}\leq C(m,\Lambda,N,p,\epsilon)\, .
 \end{gather*}
\end{corollary}

Also the improved covering arguments of Section \ref{sec_isol_sing} carry over immediately to the stationary case.
\begin{theorem}
Under the hypothesis of the previous theorem, suppose that $p$ is not an integer and $m=\floor{p}+1$, or that $m=b+2$ under the additional assumption \eqref{B}. Let $u$ be a stationary $p$-harmonic map $u:B_2(0)\to N$, where 
 \begin{gather*}
  \int_{B_2(0)}\abs{\nabla u}^p dV \leq \Lambda\, .
 \end{gather*}
Then
 \begin{gather*}
  \# \cS(u)\cap B_{1}(0)\leq C(p,\Lambda,N)\, .
 \end{gather*}
\end{theorem}

\subsection*{Acknowledgements.}The first author has been supported by NSF grant DMS-1406259. The second author has been supported by SNSF projects 149539 and 157452. The second and third authors have been partially supported by the Gruppo Nazionale per l'Analisi Matematica, la Probabilit\`a e le loro Applicazioni (GNAMPA) of the Istituto Nazionale di Alta Matematica (INdAM). 

\bibliographystyle{aomalpha}
\bibliography{p-bib}

\end{document}